%% file: main.tex
\documentclass{article}
\pdfoutput=1 
\usepackage[final, nonatbib]{neurips_2019}
\usepackage[utf8]{inputenc} % allow utf-8 input
\usepackage[T1]{fontenc}    % use 8-bit T1 fonts
\usepackage{hyperref}       % hyperlinks
\usepackage{url}            % simple URL typesetting
\usepackage{booktabs}       % professional-quality tables
\usepackage{amsfonts}       % blackboard math symbols
\usepackage{nicefrac}       % compact symbols for 1/2, etc.
\usepackage{microtype}      % microtypography
\renewcommand{\arraystretch}{1.1}

\title{Continuous-time Models\\ for Stochastic Optimization Algorithms}

\author{
Antonio Orvieto\\Department of Computer Science\\ETH Zurich, Switzerland~\thanks{Correspondence to \url{orvietoa@ethz.ch}{}}
 \And
Aurelien Lucchi\\Department of Computer Science\\ETH Zurich, Switzerland} 
\input{defs.tex}

\tcbset{boxsep=0mm,boxrule=0pt,colframe=white,arc=0mm,left=0.5mm,right=0.5mm}

\begin{document}

\maketitle

\begin{abstract}
We propose new continuous-time formulations for first-order stochastic optimization algorithms such as mini-batch gradient descent and variance-reduced methods.
We exploit these continuous-time models, together with simple Lyapunov analysis as well as tools from stochastic calculus, in order to derive convergence bounds for various types of non-convex functions. Guided by such analysis, we show that the same Lyapunov arguments hold in discrete-time, leading to matching rates. In addition, we use these models and It\^{o} calculus to infer novel insights on the dynamics of SGD, proving that a decreasing learning rate acts as time warping or, equivalently, as landscape stretching. 
% We contrast these bounds to their known equivalent in discrete-time as well as derive new bounds, both in continuous and discrete time. Our model also includes SVRG, for which we derive a linear convergence rate for the class of restricted secant functions.
\end{abstract}

\input{00_introduction}
\input{01_model}

\input{02_convergence_results}

\input{03_insights}
%\input{04_results}
\input{05_conclusion}

% Acknowledgments---Will not appear in anonymized version
\bibliographystyle{plain}
\bibliography{papers}

\input{appendix.tex}

\end{document}

%% file: defs.tex
\usepackage{amsmath}
\usepackage{amssymb}
\usepackage{mathrsfs}
\usepackage[makeroom]{cancel}
\usepackage{caption}
\usepackage{subcaption}
\usepackage{algorithm}
\usepackage{algorithmic}
\usepackage{wrapfig}
\usepackage{graphicx} 
\usepackage{tcolorbox}
\usepackage{tikz}
\usepackage{tabu}
\usepackage{float}

\DeclareMathOperator{\tr}{Tr}

\DeclareMathOperator{\var}{\mathbb{C}ov}

\DeclareMathOperator*{\argmin}{arg\,min}

\newcommand{\R}{\mathbb{R}}

\newcommand{\T}{\mathfrak{T}}

\newcommand{\E}{\mathbb{E}}
\newcommand{\I}{I}
\newcommand{\N}{\mathbb{N}}

\newcommand{\EE}{\mathcal{E}}

\newcommand{\F}{\mathcal{F}}

\usepackage{amsthm}
\usepackage{mdframed}
\newmdtheoremenv{thmapp}{Theorem}[section]
\newmdtheoremenv{propapp}{Proposition}[section]

\newtheorem{definition}{Definition}
\newtheorem{remark}{Remark}
\newtheorem{example}{Example}[section]

\newmdtheoremenv{frm-thm}{Theorem}
\newmdtheoremenv{frm-thm-app}{Theorem}[section]
\newmdtheoremenv{corollary}{Corollary}
\newmdtheoremenv{lemma}{Lemma}
\usepackage{enumitem}

\makeatletter
\newenvironment{chapquote}[2][2em]
  {\setlength{\@tempdima}{#1}%
   \def\chapquote@author{#2}%
   \parshape 1 \@tempdima \dimexpr\textwidth-2\@tempdima\relax%
   \itshape}
  {\par\normalfont\hfill--\ \chapquote@author\hspace*{\@tempdima}\par\bigskip}
\makeatother

%% file: 00_introduction.tex
% !TEX root = aistats18_svrg.tex

\section{Introduction}
We consider the problem of finding the minimizer of a smooth non-convex function $f:\R^d\to\R$: $x^*:= \argmin_{x\in \R^d} f(x)$. We are here specifically interested in a finite-sum setting which is commonly encountered in machine learning and where $f(\cdot)$ can be written as a sum of individual functions over datapoints. In such settings, the optimization method of choice is mini-batch Stochastic Gradient Descent (MB-SGD) which simply iteratively computes stochastic gradients based on averaging from sampled datapoints. The advantage of this approach is its cheap per-iteration complexity  which is independent of the size of the dataset. This is of course especially relevant given the rapid growth in the size of the datasets commonly used in machine learning applications. However, the steps of MB-SGD have a high variance, which can significantly slow down the speed of convergence~\cite{ghadimi2013stochastic,kushner2003stochastic}. In the case where $f(\cdot)$ is a strongly-convex function, SGD with a decreasing learning rate achieves a sublinear rate of convergence in the number of iterations, while its deterministic counterpart (i.e. full Gradient Descent, GD) exhibits a linear rate of convergence. 

There are various ways to improve this rate. The first obvious alternative is to systematically increase the size of the mini-batch at each iteration: ~\cite{friedlander2012hybrid} showed that a controlled increase of the mini-batch size yields faster rates of convergence. An alternative, that has become popular recently, is to use variance reduction (VR) techniques such as
SAG~\cite{roux2012stochastic}, SVRG~\cite{johnson2013accelerating},  SAGA~\cite{defazio2014saga}, etc. The high-level idea behind such algorithms is to re-use past gradients on top of MB-SGD in order to reduce the variance of the stochastic gradients. This idea leads to faster rates: for general $L$-smooth objectives, both SVRG and SAGA find an $\epsilon$-approximate stationary point\footnote{A point $x$ where $\|\nabla f(x)\|\le \epsilon$.} in $\mathcal{O}\left(Ln^{2/3}/\epsilon\right)$ stochastic gradient computations~\cite{allen2016variance, reddi2016stochastic}, compared to the $\mathcal{O}\left(Ln/\epsilon\right)$ needed for GD~\cite{nesterov2018lectures} and the $\mathcal{O}\left(L/\epsilon^2\right)$ needed for MB-SGD~\cite{ghadimi2013stochastic}. As a consequence, most modern state-of-the-art optimizers designed for general smooth objectives (Natasha~\cite{allen2017natasha}, SCSG~\cite{lei2017non}, Katyusha~\cite{allen2017katyusha}, etc) are based on such methods. The optimization algorithms discussed above are typically analyzed in their discrete form. One alternative that has recently become popular in machine learning is to view these methods as continuous-time processes. By doing so, one can take advantage of numerous tools from the field of differential equations and stochastic calculus.
%\textit{Unfortunately, the analysis of the methods above can be complex, since it often relies on heavy algebraic manipulations which reveal little about the underlying mechanisms of the proof} and of the algorithm. One alternative that has recently become popular in the machine learning community is to use tools from the field of differential equations and stochastic calculus.
This has led to new insights about non-trivial phenomena in non-convex optimization~\cite{mandt2016variational,jastrzkebski2017three,su2014differential} and has allowed for more compact proofs of convergence for gradient methods~\cite{shi2019acceleration, mertikopoulos2018convergence, krichene2017acceleration}. This perspective appears to be very fruitful, since it also has led to the development of new discrete algorithms~\cite{zhang2018direct,betancourt2018symplectic,xu2018accelerated,xu2018continuous}. Finally, this connection goes beyond the study of algorithms, and can be used for neural network architecture design~\cite{ciccone2018nais,chen2018}. 

This success is not surprising, given the impact of continuous-time models in various scientific fields including, e.g., mathematical finance, where these models are often used to get closed-form solutions for derivative prices that are not available for discrete models (see e.g. the celebrated Black-Scholes formula~\cite{black1973pricing}, which is derived from It\^{o}'s lemma~\cite{ito1944109}). Many other success stories come from statistical physics~\cite{einstein1905motion}, biology~\cite{goel2016stochastic} and engineering.
Nonetheless, an important question, which has encouraged numerous debates (see e.g.~\cite{wigner1990unreasonable}), is about the reason behind the effectiveness of continuous-time models. In optimization, this question is partially addressed for deterministic accelerated methods by the works of~\cite{wilson2016lyapunov,betancourt2018symplectic,shi2019acceleration} that provide a link between continuous and discrete time. However, we found that this problem has received less attention in the context of stochastic non-convex optimization and does not cover recent developments such as~\cite{johnson2013accelerating}. We therefore focus on the latter setting for which we provide detailed comparisons and analysis of continuous- and discrete-time methods.
The paper is organized as follows:
\begin{enumerate}[leftmargin=0.3in]
    \setlength\itemsep{0em}
    \item In Sec.~\ref{sec:models} we build new continuous-time models for SVRG and mini-batch SGD --- which include the effect of decaying learning rates and increasing batch-sizes. We show existence and uniqueness of the solution to the corresponding stochastic differential equations.
    \item In Sec.~\ref{sec:non-asy} we derive novel and interpretable non-asymptotic convergence rates for our models, using the elegant machinery provided by stochastic calculus. We focus on various classes of non-convex functions relevant for machine learning (see list in Sec.~\ref{sec:rates}).
    \item In Sec.~\ref{sec:non-asy-discrete} we complement each of our rates in continuous-time with equivalent results for the algorithmic counterparts, using the same Lyapunov functions. This shows an \textit{algebraic equivalence} between continuous and discrete time and proves the effectiveness of our modeling technique. To the best of our knowledge, most of these rates (in full generality) are novel~\footnote{We derive these rates in App.~\ref{sec:proofs_dt} and summarize them in Tb.~\ref{tb:summary_rates_GD}.}. 
    \item In Sec.~\ref{sec:time_stretch} we provide a new interpretation for the distribution induced by SGD with decreasing stepsizes based on the {\O}ksendal's time change formula --- which reveals an underlying time warping phenomenon that can be used for designing Lyapunov functions.
    \item In Sec.~\ref{sec:land-stretch} we provide a dual interpretation of this last phenomenon as landscape stretching.
\end{enumerate}
 
 At a deeper level, our work proves that continuous-time models can adequately guide the analysis of stochastic gradient methods and provide new thought-provoking perspectives on their dynamics.

%% file: 01_model.tex
\vspace{-1mm}
\section{Unified models of stochastic gradient methods}
\label{sec:models}
\vspace{-1mm}
Let $\{f_i\}_{i=1}^N$ be a collection of functions s.t. $f_i : \mathbb{R}^d \to \mathbb{R}$ for any $i \in [N]$  and $f(\cdot):= \frac{1}{N}\sum_{i = 1}^N f_i(\cdot)$. 
In order to minimize $f(\cdot)$,
first-order stochastic optimization algorithms rely on some noisy (but usually unbiased) estimator $\mathcal{G}(\cdot)$ of the gradient $\nabla f(\cdot)$. In its full generality, Stochastic Gradient Descent (SGD) builds a sequence of estimates of the solution $x^*$ in a recursive way:
\begin{equation*}
\tag{SGD}
x_{k+1} = x_{k} - \eta_k \mathcal{G}\left(\{x_i\}_{0\le i\le k},k\right),
\label{eq:SGD}
\end{equation*}
where $(\eta_k)_{k\ge0}$ is a non-increasing deterministic sequence of positive numbers called the \textit{learning rates sequence}. Since $\mathcal{G}(x_k, k)$ is stochastic, $\{x_k\}_{k\ge0}$ is a stochastic process on some countable probability space $(\Omega,\mathcal{F},\mathbb{P})$. 
Throughout this paper, we denote by $\{\mathcal{F}_k\}_{k\ge 0}$ the natural filtration induced by $\{x_k\}_{k\ge 0}$; by $\E$ the expectation over all the information $\mathcal{F}_\infty$ and by $\E_{\mathcal{F}_{k}}$ the conditional expectation given the information at step $k$. We consider the two following popular designs for $\mathcal{G}(\cdot)$.
\vspace{-1mm}
\paragraph{i) MB gradient estimator.} The mini-batch gradient estimator at iteration $k$ is $\mathcal{G}_{\text{MB}}(x_k,k) := \frac{1}{b_k}\sum_{i_k \in \Omega_k}\nabla f_{i_k}(x_k)$, where $b_k:= |\Omega_k|$ and the elements of $\Omega_k$ (the \textit{mini-batch}) are sampled at each iteration $k$ independently, uniformly and with replacement from $[N]$. Since $\Omega_k$ is random, $\mathcal{G}_{\text{MB}}(x)$ is a random variable with conditional (i.e. taking out randomness in $x_k$) mean and covariance
\begin{equation}
  \E_{\F_{k-1}}\left[\mathcal{G}_{\text{MB}}(x_k,k)\right] = \nabla f(x_k),\quad \quad \var_{\F_{k-1}}\left[\mathcal{G}_{\text{MB}}(x_k,k)\right] = \frac{\Sigma_{\text{MB}}(x_k)}{b_k}, \label{eqn:STOC_MB_2}
\end{equation}
where $\Sigma_{\text{MB}}(x):=\frac{1}{N}\sum_{i=1}^N \left(\nabla f(x) -\nabla f_{i}(x)\right)\left(\nabla f(x) -\nabla f_{i}(x)\right)^T$ is the one-sample covariance.
\vspace{-1mm}
\paragraph{ii) VR gradient estimator.}
The basic idea of the original SVRG algorithm introduced in~\cite{johnson2013accelerating} is to compute the full gradient at some chosen \textit{pivot} point and combine it with stochastic gradients computed at subsequent iterations. Combined with mini-batching~\cite{reddi2016stochastic}, this gradient estimator is:
$$\mathcal{G}_{\text{VR}}(x_k,\tilde{x}_k,k) := \frac{1}{b_k}\sum_{i_k \in \Omega_k}\nabla f_{i_k}(x_k)- \nabla f_{i_k}(\tilde{x}_k)+\nabla f(\tilde{x}_k),$$
where $\tilde{x}_k\in \{x_0,x_1,\dots,x_{k-1}\}$ is the pivot used at iteration $k$. This estimator is unbiased, i.e. $\E_{\F_{k-1}}\left[\mathcal{G}_{\text{VR}}(x_k,\tilde{x}_k,k)\right]  = \nabla f(x_k)$. Its covariance is $\var_{\F_{k-1}}\left[\mathcal{G}_{\text{VR}}(x_k,\tilde{x}_k, k)\right]=\frac{\Sigma_{\text{VR}}(x_k,\tilde{x}_k)}{b_k}$ with\\
\begin{small}$\Sigma_{\text{VR}}(x,y):=\frac{1}{N}\sum_{i=1}^N\left(\nabla f_{i}(x)- \nabla f_{i}(y)+\nabla f(y)-\nabla f(x)\right)\left(\nabla f_{i}(x)- \nabla f_{i}(y)+\nabla f(y)-\nabla f(x)\right)^T$\end{small}.
%------BEGIN SIMPLE

%We make a simplification similar to the one made earlier for the mini-batch gradient estimate.

%\textbf{Simplification GaussVR}
%$$\mathcal{G}_{\text{VR}}(x_k)\sim\mathcal{N}(\nabla f(x_k), \Sigma_{\text{VR}}(x_k,\tilde{x}_k)).$$

%------END SIMPLE

%%%%%%%%%%%%%%%%%%%%%%%%%%%%%%%%%%%%%%%%%%%%%%%%%%%%%%%%%%%%%%%%%%%%%%%%%%%%%%%
%%%%%%%%%%%%%%%%%%%%%%%%%%%%%%%%%%%%%%%%%%%%%%%%%%%%%%%%%%%%%%%%%%%%%%%%%%%%%%%

\subsection{Building the perturbed gradient flow model}
We take inspiration from~\cite{li2015stochastic} and~\cite{he2018differential} and build continuous-time models for SGD with either the MB or the SVRG gradient estimators. The procedure has three steps.
\begin{enumerate}[label=(\text{S}\arabic*),leftmargin=*]
    \item We first define the \textit{discretization stepsize} $h:=\eta_0$ --- this variable is essential to provide a link between continuous and discrete time. We assume it to be fixed for the rest of this subsection. Next, we define the \textit{adjustment-factors sequence} $(\psi_k)_{k\ge 0}$ s.t. $\psi_k=\eta_k/ h$ (cf. Eq.~9 in~\cite{li2015stochastic}). In this way --- we decouple the two information contained in $\eta_k$: $h$ controls the overall size of the learning rate and $\psi_k$ handles its variation\footnote{A popular choice (see e.g. ~\cite{moulines2011non}) is $\eta_k = C k^{-\alpha}$, $\alpha\in[0,1]$. Here, $h=C$ and $\psi_k = k^{-\alpha}\in[0,1]$.} during training. 

    \item Second, we write SGD as $x_{k+1} = x_k - \eta_k (\nabla f(x_k) + V_k)$,
where the error $V_k$ has mean zero and covariance $\Sigma_k$. Next, let $\Sigma^{1/2}_k$ be the principal square root\footnote{The unique positive semidefinite matrix such that $\Sigma_k = \Sigma^{1/2}_k\Sigma^{1/2}_k$.} of $\Sigma_k$, we can write SGD as
\begin{equation}
    \tag{PGD}
    x_{k+1} = x_k - \eta_k \nabla f(x_k) -\eta_k \Sigma^{1/2}_k Z_k,
\end{equation}

where $Z_k$ is a random variable with zero mean and unit covariance\footnote{Because $\Sigma^{1/2}_k \ Z_k$ has the same distribution as $V_k$, conditioned on $x_k$.}. In order to build \textit{simple} continuous-time models, we assume that each $Z_k$ is Gaussian distributed: $Z_k\sim\mathcal{N}(0_d,\I_d)$. To highlight this assumption, we will refer to the last recursion as Perturbed Gradient Descent (PGD)~\cite{daneshmand2018escaping}. In Sec.~\ref{sec:gauss} we motivate why this assumption, which is commonly used in the literature~\cite{li2015stochastic}, is not restrictive for our purposes. By plugging in either $\Sigma_k = \Sigma_{\text{MB}}(x_k)/b_k$ or $\Sigma_k =\Sigma_{\text{VR}}(x_k,\tilde{x}_k)/b_k$, we get a discrete model for SGD with the MB or VR gradient estimators.

%Usually, in SVRG-like algorithms, $\tilde{x}_k$ is changed every $m$ iterations. The new pivot is typically chosen to be the last available iterate. The gradient estimate is the most precise at the first iteration after the change of pivot, and becomes more stale in later iterations. This is why the pivot should be changed as frequently as possible and the convergence rate is directly linked to the value of $m$ (see e.g. Theorem 1 in \cite{johnson2013accelerating} or Theorem 5.1 in \cite{allen2016variance}). 

%In continuous time, the event "change of pivot" cannot be modeled easily without splitting the integration into multiple intervals. In order to simplify the analysis, we will assume --- specifically for the continuous time model --- that \textit{the pivot is always $m$ steps away from the current iterate}: $\tilde x_k = x_{k-m}$. This will therefore yield a worst-case analysis of SVRG.

\item Finally, we lift these PGD models to continuous time. The first step is to rewrite them using $\psi_k$:
\begin{equation*}
    \tag{MB-PGD}
    x_{k+1} = x_k -\underbrace{\psi_k \nabla f(x_k)}_{\text{ adjusted gradient drift}} h  +  \underbrace{\psi_k\sqrt{h/b_k} \ \sigma_{\text{MB}}(x_k)}_{\text{adjusted mini-batch volatility}} \ \sqrt{h} Z_k \ \ \ \ \ \ \ \ 
\end{equation*}

\begin{equation*}
    \tag{VR-PGD}
    x_{k+1} = x_k -\underbrace{\psi_k \nabla f(x_k)}_{\text{adjusted gradient drift}} h + \underbrace{\psi_k\sqrt{h/b_k} \ \sigma_{\text{VR}} (x_k,x_{k-\xi_k})}_{\text{adjusted variance-reduced volatility}} \ \sqrt{h} Z_k
\end{equation*}

where $\sigma_{\text{MB}}(x) := \Sigma^{1/2}(x)$, $\sigma_{\text{VR}}(x,y) := \Sigma_{\text{VR}}^{1/2}(x,y)$ and $\xi_k\in[k]$ quantifies the pivot staleness. Readers familiar with stochastic analysis might recognize that MB-PGD and VR-PGD are the steps of a numerical integrator (with stepsize $h$) of an SDE and of an SDDE, respectively. For convenience of the reader, we give an hands-on introduction to these objects in App.~\ref{sec:stoc_cal}.
\end{enumerate}

The resulting continuous-time models, which we analyse in this paper, are
\begin{tcolorbox}
\vspace{-4mm}
\begin{equation*}
\tag{MB-PGF}
dX(t) = -\psi(t)\nabla f(X(t)) \ dt + \psi(t)\sqrt{h/b(t)} \ \sigma_{\text{MB}}(X(t)) \ dB(t) \ \ \ \ \ \ \ \ \ \ \ \ \ \ \ \ \ \ \ \ 
\end{equation*}
\begin{equation*}
\tag{VR-PGF}
dX(t) = -\psi(t)\nabla f(X(t)) \ dt + \psi(t)\sqrt{h/b(t)} \ \sigma_{\text{VR}}(X(t),X(t-\xi(t))) \ dB(t)
\end{equation*}
\vspace{-4mm}
\end{tcolorbox}
where 
\begin{itemize}[leftmargin=*]
\setlength\itemsep{0.1em}
\item  $\xi:\R_+\to[0,\T]$, the \textit{staleness function}, is s.t. $\xi(h k)=\xi_k$ for all $k\ge 0$;
\item  $\psi(\cdot)\in\mathcal{C}^1(\R_+,[0,1])$, the \textit{adjustment function}, is s.t. $\psi(h k)=\psi_k$ for all $k\ge 0$ and $\frac{d\psi(t)}{dt} \le 0$;
\item $b(\cdot) \in\mathcal{C}^1(\R_+,\R_+)$, the \textit{mini-batch size function} is s.t. $b(hk)=b_k$ for all $k\ge0$ and $b(t)\ge1$;
\item $\{B(t)\}_{t\ge 0}$ is a $d-$dimensional Brownian Motion on some filtered probability space.
\end{itemize}

We conclude this subsection with some important remarks and clarifications on the procedure above.
\vspace{-1mm}
\label{sec:gauss}
\paragraph{On the Gaussian assumption.}In (S2) we assumed that $Z_k$ is Gaussian distributed. If the mini-batch size $b_k$ is large enough and the gradients are sampled from a distribution with finite variance, then the assumption is sound: indeed, by the Berry–Esseen Theorem (see e.g.~\cite{durrett2010probability}), $Z_k$ approaches $\mathcal{N}(0_d,\I_d)$ in distribution with a rate $\mathcal{O}\left(1/\sqrt{b_k}\right)$. However, if $b_k$ is small or the underlying variance is unbounded, the distribution of $Z_k$ has heavy tails~\cite{simsekli2019tail}.
Nonetheless, in the large-scale optimization literature, the gradient variance is generally assumed to be bounded~(see e.g.~\cite{ghadimi2013stochastic},~\cite{bottou2018optimization}) --- hence, we keep this assumption, which is practical and reasonable for many problems (likewise assumed in the related literature~\cite{raginsky2012continuous,mertikopoulos2018convergence,krichene2017acceleration,li2015stochastic,li2017convergence}). Also, taking a different yet enlightening perspective, it is easy to see that (see Sec.~4 of~\cite{bottou2018optimization}), if one cares only about expected convergence guarantees --- only the first and the second moments of the stochastic gradients have a quantitative effect on the rate.
%The discussion in the next paragraph also motivates the assumption.
\vspace{-1mm}
\paragraph{Approximation guarantees.}Recently,~\cite{hu2017diffusion,li2017convergence} showed that for a special case of MB-PGF ($\psi_k = 1$, and $b_k$ constant), its solution $\{X(t)\}_{0\le t\le T}$ compares to SGD as follows: let $K = \lfloor T/h \rfloor$ and consider the iterates $\{x_k\}_{k\in[K]}$ of mini-batch SGD (i.e. \textit{without} Gaussian assumption) with fixed learning rate $h$. Under mild assumptions on $f(\cdot)$, there exists a constant $C$ (independent of $h$) such that $\|\E[x_k]-\E[X(kh)]\|\le C h$ for all $k\in[K]$. Their proof argument relies on semi-group expansions of the solution to the Kolmogorov backward equation, and can be adapted to provide a similar result for our (more general) equations. However, this approach to motivate the continuous-time formulation is\textit{ very limited} --- as $C$ depends exponentially on $T$~(see also~\cite{shi2019acceleration}). Nonetheless, under strong-convexity, some uniform-in-time (a.k.a. \textit{shadowing}) results were recently derived in~\cite{orvietoshadowing,feng2019uniform}. In this paper, we take a different approach (similarly to~\cite{shi2019acceleration} for deterministic methods)  and provide instead matching convergence rates in continuous and in discrete time using the same Lyapunov function. We note that this is still a very strong indication of the effectiveness of our model to study SGD, since it shows an \textit{algebraic equivalence} between the continuous and the discrete case.
\vspace{-1mm}
\paragraph{Comparison to the "ODE method".}A powerful technique in stochastic approximation~\cite{kushner2003stochastic} is to study SGD through the \textit{deterministic} ODE $\dot X = -\nabla f(X)$. A key result is that SGD, with decreasing learning rate under the Robbins Monro~\cite{robbins1951stochastic} conditions, behaves like this ODE in the limit. Hence the ODE can be used to characterize the \emph{asymptotic} behaviour of SGD. In this work we instead take inspiration from more recent literature~\cite{li2017convergence} and build \textit{stochastic} models which include the effect of a decreasing learning rate into the drift and the volatility coefficients through the adjustment function $\psi(\cdot)$. This allows, in contrast to the ODE method\footnote{This method is instead suitable to assess almost sure convergence and convergence in probability, which are not considered in this paper for the sake of delivering convergence rates for population quantities.}, to provide \emph{non-asymptotic} arguments and convergence rates.
\vspace{-1mm}
\paragraph{Local minima width.} Our models confirm, as noted in~\cite{jastrzkebski2017three}, that the ratio of (initial) learning rate $h$ to batch size $b(t)$ is a determinant factor of SGD dynamics. Compared to~\cite{jastrzkebski2017three}, our model is more general: indeed, we will see in Sec.~\ref{sec:land-stretch} that the adjustment function also plays a fundamental role in determining the width of the final minima --- since it acts like a "function stretcher".
\vspace{-1mm}
\subsection{Existence and uniqueness }
\vspace{-1mm}
Prior works that take an approach similar to ours~\cite{krichene2015accelerated,he2018differential,mertikopoulos2018convergence}, assume the one-sample volatility $\sigma(\cdot)$ to be Lipschitz continuous. This makes the proof of existence and uniqueness straightforward (cf. a textbook like~\cite{mao2007stochastic}), but we claim such assumption is not trivial in our setting where $\sigma(\cdot)$ is data-dependent. Indeed, $\sigma(\cdot)$ is the result of a square root operation on the gradient covariance --- and the square root function is not Lipschitz around zero. App.~\ref{sec:ex_uni} is dedicated to a rigorous proof of existence and uniqueness, which is verified under the following condition:

\textbf{(H)} \ \ \ \ \ \ \ \ Each $f_i(\cdot)$ is $\mathcal{C}^3$, with bounded third derivative and $L$-smooth.

This hypothesis is arguably not restrictive as it is usually satisfied by many loss functions encountered in machine learning. As a result, under \textbf{(H)}, with probability $1$ the realizations of the stochastic processes $\{f(X(t))\}_{t>0}$ and $\{X(t)\}_{t>0}$ are continuous functions of time.
% In Theorem \ref{thm:STOC_ex_uni_MB} we show that such assumption can be dropped. This degree of carefulness might not be needed to model mini-batch SGD (one can always consider an isotropic and constant upper bound on the noise variance and provide a worst-case analysis), but is necessary in order to analyze meaningful models of variance reduced and adaptive methods: in such models the diffusion changes as function of the iteration and gets smaller and smaller as the algorithm is approaching the origin, making the answer to the question of existence and uniqueness non-trivial. Indeed, Appendix \ref{sec:ex_uni} is dedicated t

%% file: 02_convergence_results.tex
% !TEX root = aistats18_svrg.tex

\section{Matching convergence rates in continuous and discrete time} 
\label{sec:rates}
\vspace{-1mm}
Even though in optimization, convex functions are central objects of study, many interesting objectives found in machine learning are non-convex. However, most of the time, such functions still exhibit some regularity. For instance,~\cite{hardt2016gradient} showed that linear LSTMs induce weakly-quasi-convex objectives.  

\textbf{(H\begin{tiny}{WQC}\end{tiny})} \ \ \  $f(\cdot)$ is $\mathcal{C}^1$ and exists $\tau>0$ and $x^\star$ s.t. $\langle\nabla f(x),x-x^{\star}\rangle \ge \tau(f(x)-f(x^{\star}))$ for all $x\in\R^d$.

%\\ \hspace*{1.25cm} 

Intuitively, \textbf{(H\begin{tiny}{WQC}\end{tiny})} requires the negative gradient to be always aligned with the direction of a global minimum $x^{\star}$. Convex differentiable functions are weakly-quasi-convex (with $\tau=1$), but the WQC class is richer and actually allows functions to be locally concave.
Another important class of problems (e.g., under some assumptions, matrix completion~\cite{sun2016guaranteed}) satisfy the Polyak-\L{}ojasiewicz property, which is the weakest known sufficient condition for GD to achieve linear convergence~\cite{polyak1963gradient}.   

\textbf{(H\begin{tiny}{P\L{}}\end{tiny})} \ \ \ \ \ \ $f(\cdot)$ is $\mathcal{C}^1$ and there exists $\mu>0$ s.t. $\|\nabla f(x)\|^2\ge2\mu(f(x)-f(x^\star))$ for all $x\in\R^d$.

One can verify that if $f(\cdot)$ is strongly-convex, then it is P\L{}. However, P\L{} functions are not necessarily convex. What's more, a broad class of problems (dictionary learning~\cite{arora2015simple}, phase retrieval~\cite{chen2015solving}, two-layer MLPs~\cite{li2017convergence}) are related to a stronger condition: the restricted-secant-inequality~\cite{zhang2016new}.

\textbf{(H\begin{tiny}{RSI}\end{tiny})} \ \ \ \ \ $f(\cdot)$ is $\mathcal{C}^1$ and there exists $\mu>0$ s.t. $\langle\nabla f(x),x-x^*\rangle\ge\frac{\mu}{2}\|x-x^*\|^2$ for all $x\in\R^d$.

In~\cite{karimi2016linear} the authors prove strong-convexity $\Rightarrow$ \textbf{(H\begin{tiny}{RSI}\end{tiny})} $\Rightarrow$ \textbf{(H\begin{tiny}{P\L{}}\end{tiny})} (with different constants).

% There exist many conditions which are stronger than P\L{} and do not imply convexity, one of them being the restricted-secant-inequality, introduced in \cite{zhang2013gradient}.
% \begin{definition}[$\mu$-restricted-secant-inequality]
% Assume $f\in\mathcal{C}^1(\R^d,\R)$ and let $\mu>0$. $f\in\mu$-RSI$(x^{\star})$ if and only if for all $x\in \R^d$
% $$\langle\nabla f(x),x-x^{\star}\rangle \ge \mu\|x-x^{\star}\|^2.$$
% \label{def:FUN_RSI}
% \end{definition}
% \vspace{-5mm}
% In \cite{karimi2016linear} the authors show that if $f$ is $L$-smooth (definition in Section \ref{sec:models}) and $f\in\mu$-RSI$(x^{\star})$, then $f\in\frac{\mu}{L}$-P\L{}$(f(x^{\star}))$.

%%%%%%%%%%%%%%%%%%%%%%%%%%%%%%%%%%%%%%%%%%%%%%%%%%%%%%%%%%%%%%%%%%%%%%%%%%%%%%%%
%%%%%%%%%%%%%%%%%%%%%%%%%%%%%%%%%%%%%%%%%%%%%%%%%%%%%%%%%%%%%%%%%%%%%%%%%%%%%%%%

\subsection{Continuous-time analysis}
\label{sec:non-asy}
\vspace{-1mm}
First, we derive non-asymptotic rates for MB-PGF. For convenience, we define $\varphi(t) := \int_0^t \psi(s) ds$, which plays a fundamental role (see Sec.~\ref{sec:time_stretch}). As~\cite{mertikopoulos2018convergence, krichene2017acceleration}, we introduce a bound on the volatility.

\textbf{(H$\boldsymbol{\sigma}$)} $\sigma^2_*:=\sup_{x\in\R^d}\|\sigma_{\text{MB}}(x)\sigma_{\text{MB}}(x)^T\|_s<\infty$, where $\| \cdot \|_s$ denotes the spectral norm.

% For instance, consider the well-studied gradient flow ODE model for (noiseless) Gradient Descent: $\dot X = -\nabla f(X)$. The Lyapunov method consists in finding a "time-expanding" energy function $\EE(x,t)$ such that, once the dynamics is plugged in, $\frac{d}{dt}\EE(X(t),t)\le 0$. For instance, assuming $f$ is convex, we can pick $\EE(x,t) = t(f(x)-f(x^{\star})) + \frac{1}{2}\|x-x^{\star}\|^2$ (notice the proportionality of the first term with $t$). Then, using convexity, $\frac{d}{dt}\EE(X(t),t) \le f(X(t))-f(x^{\star}) - \langle\nabla f(X(t)),X(t)-x^{\star}\rangle\le 0$. Next, $\frac{d}{dt}\EE(X(t),t) \le 0$ implies $\EE(X(t),t)\le\EE(X(0),0) =\|x_0-x^{\star}\|^2$. Hence, we get the well-known sublinear rate $f(X(t))-f(x^{\star})\le \mathcal{O}(1/t)$. In the SDE setting, a similar argument can be carried out using the diffusion operator (see Eq. \ref{eqn:SDE_diffusion_op} in App. \ref{sec:stoc_cal}) instead of the derivative.

\begin{frm-thm} Assume \textbf{(H)}, \textbf{(H$\boldsymbol{\sigma}$)}. Let $t>0$ and $\tilde t\in [0,t]$ be a random time point with distribution $\frac{\psi(\tilde t)}{\varphi(t)}$ for $\tilde t\in[0,t]$ (and $0$ otherwise). The solution to \underline{MB-PGF} is s.t.
$$\E\left[\|\nabla f(X(\tilde t))\|^2\right] \le \frac{f(x_0)-f(x^{\star})}{\varphi(t)}+ \frac{h \ d \ L \ \sigma_*^2}{2 \ \varphi(t)}\int_0^t\frac{\psi(s)^2}{b(s)}ds.$$
\label{prop:CONV_GF_smooth}
\vspace{-3mm}
\end{frm-thm}

\begin{proof} We use the energy function $\EE(x,t) := f(x)-f(x^{\star})$. Details in App.~\ref{sec:proofs_MB_PGF}.
\end{proof}
\begin{frm-thm}
Assume \textbf{(H)}, \textbf{(H$\boldsymbol{\sigma}$)}, \textbf{(H\begin{tiny}{WQC}\end{tiny})}. Let $\tilde t$ be as in Thm.~\ref{prop:CONV_GF_smooth}. The solution to \underline{MB-PGF} is s.t.
\begin{equation}
\tag{W1}
    \E \left[f(X(\tilde t)) -f(x^{\star})\right]\le \frac{\|x_0-x^{\star}\|^2}{2 \ \tau \ \varphi(t) }+ \frac{h \ d \ \sigma^2_*}{2 \ \tau \ \varphi(t)}\int_0^t\frac{\psi(s)^2}{b(s)}ds \ \ \ \ \ \ \ \ \ \ \ \ \ \ \ \ \ \ \ \ \  
\end{equation}
\begin{equation}
    \tag{W2}
    \E\left[(f(X(t))-f(x^{\star}))\right]  \le\frac{\|x_0-x^{\star}\|^2}{2 \ \tau \ \varphi(t) }+ \frac{h \ d \ \sigma^2_*}{2 \ \tau \ \varphi(t) }\int_0^t(L \ \tau \ \varphi(s)+1)\frac{\psi(s)^2}{b(s)}ds.
\end{equation}
    \label{prop:CONV_GF_WQC}
    \vspace{-2mm}
\end{frm-thm}
\begin{proof} We use the energy functions $\EE_1, \EE_2$ s.t. $\EE_1(x) := \frac{1}{2}\|x-x^{\star}\|^2$ and $\EE_2(x,t) := \tau \varphi(t) (f(x))-f(x^{\star})) + \frac{1}{2}\|x-x^{\star}\|^2$ for (W1) and (W2), respectively. Details in App.~\ref{sec:proofs_MB_PGF}.
\end{proof}

%Note that $L$ does not appear in Eq. W1. However, smoothness is required for existence and uniqueness of the solution.

\begin{frm-thm}
Assume \textbf{(H)}, \textbf{(H$\boldsymbol{\sigma}$)}, \textbf{(H\begin{tiny}{P\L{}}\end{tiny})}. The solution to \underline{MB-PGF} is s.t.
$$\E[f(X(t))-f(x^{\star})] \le e^{-2\mu \varphi(t)}(f(x_0)-f(x^{\star})) + \frac{h \ d \ L \ \sigma^2_* }{2} \int_0^t\frac{\psi(s)^2}{b(s)} e^{-2\mu (\varphi(t)-\varphi(s))}ds.$$
\label{prop:CONV_GF_PL}
\vspace{-3mm}
\end{frm-thm} 
\begin{proof} We use the energy function $\EE(x,t) := e^{2\mu \varphi(t)}(f(x)-f(x^{\star}))$. Details in App.~\ref{sec:proofs_MB_PGF}.
\end{proof}

\vspace{-1mm}
\paragraph{Decreasing mini-batch size.} From Thm.~\ref{prop:CONV_GF_smooth},~\ref{prop:CONV_GF_WQC},~\ref{prop:CONV_GF_PL}, it is clear that, as it is well known~\cite{bottou2018optimization,balles2016coupling}, a simple way to converge to a local minimizer is to pick $b(\cdot)$ increasing as a function of time. However, this corresponds to dramatically increasing the complexity in terms of gradient computations. In continuous-time, we can account for this by introducing $\beta(t) = \int_0^tb(s)ds$, proportional to the number of computed gradients at time $t$. The complexity in number of gradient computations can be derived by substituting into the final rate the \textit{new time variable} $\beta^{-1}(t)$ instead of $t$. As we will see in Thm.~\ref{thm:time-change}, this concept extends to a more general setting and leads to valuable insights.

\vspace{-1mm}
\paragraph{Asymptotic rates.} Another way to guarantee convergence to a local minimizer is to decrease $\psi(\cdot)$. In App.~\ref{sec:proofs_MB_PGF_decrease} we derive asymptotic rates for $\psi(t) = \mathcal{O}(t^{-a})$ and report the results in Tb.~\ref{tab:asy}. The results \textit{match exactly} the corresponding know rates for SGD, stated under stronger assumptions in~\cite{moulines2011non}. As for increasing $b(\cdot)$, decreasing $\psi(\cdot)$ can also be seen as performing a time warp (see Thm.~\ref{thm:time-change}).

\vspace{-1mm}
\paragraph{Ball of convergence.} For $\psi(t)=1$, the sub-optimality gap derived in App.~\ref{sec:proofs_MB_PGF_ball} \textit{matches}~\cite{bottou2018optimization}.

\begin{table}
\caption{Asymptotic rates for MB-PGF under $\psi(t) = \mathcal{O}(t^{-a})$ in the form $\mathcal{O}(t^{-\beta})$. $\beta$ shown in the table as a function of $a$. "$\boldsymbol{\sim}$" indicates randomization of the result. Rates \textit{match} with Tb.~1 in~\cite{moulines2011non}.}
\vspace{1mm}
\begin{center}
\begin{tabular}{|l | c  | c | c | c |}
\hline
$a$ & \small{\textbf{(H)}, \textbf{(H$\boldsymbol{\sigma}$)}, \textbf{(H\begin{tiny}{P\L{}}\end{tiny})}} & \small{\textbf{(H)}, \textbf{(H$\boldsymbol{\sigma}$)}, \textbf{(H\begin{tiny}{WQC}\end{tiny})}} & \small{\textbf{($\boldsymbol{\sim}$)}, \textbf{(H)}, \textbf{(H$\boldsymbol{\sigma}$)}, \textbf{(H\begin{tiny}{WQC}\end{tiny})}} & \small{\textbf{($\boldsymbol{\sim}$)}, \textbf{(H)}, \textbf{(H$\boldsymbol{\sigma}$)}}\\
         & Cor. \ref{cor:CONV_GF_PL_cor} & Cor. \ref{cor:CONV_GF_WQC_cor} & Cor. \ref{cor:CONV_GF_WQC_cor}
          & Cor. \ref{cor:CONV_GF_smooth_cor}\\
\hline
(0 \ \ \ \ , \ 1/2) & $a$ & $\times$ & $a$ & $a$\\
(1/2\ \ , \ 2/3) & $a$ & $2a-1$ & $1-a$ & $1-a$\\
(2/3\ \ , \ \ \ \ 1) & $a$ & $1-a$ & $1-a$ & $1-a$\\
\hline
\end{tabular}
\end{center}
\label{tab:asy}
\end{table}

In contrast to $\mathcal{G}_{\text{MB}}(\cdot)$, ~\cite{allen2016variance,allen2016improved} have shown that significant speed-ups are hard to obtain from parallel gradient computations (i.e. for $b(t)>1$) using $\mathcal{G}_{\text{VR}}(\cdot)$~\footnote{See e.g. Thm. 7 in~\cite{reddi2016stochastic} for a counterexample.}. Also, our results for MB-PGF as well as prior work~\cite{zhang2013gradient,allen2016variance,allen2016improved,reddi2016stochastic} suggest that linear rates can only be obtained with $\psi(t)=1$. Hence, for our analysis of VR-PGF, we focus on the case $b(t)=\psi(t)=1$. The following result, in the spirit of~\cite{johnson2013accelerating,allen2016improved}, relates to the so-called \textit{Option II} of SVRG.

\begin{frm-thm}
Assume \textbf{(H)}, \textbf{(H\begin{tiny}{RSI}\end{tiny})} and choose $\xi(t) = t-\sum_{j=1}^\infty\delta(t-j\T)$~(sawtooth wave), where $\delta(\cdot)$ is the Dirac delta. Let $\{X(t)\}_{t\ge0}$ be the solution to \underline{VR-PGF} with additional jumps at times $(j\T)_{j\in\N}$: we  pick $X(j\T+\T)$ uniformly in $\{X(s)\}_{j\T\le s< (j+1)\T}$. Then,
\vspace{-2mm}
$$\E[\|X(j\T)-x^\star\|^2]=\left(\frac{2 h L^2\T + 1}{\T(\mu-2hL^2)}\right)^j \|x_0-x^*\|^2.$$
\label{prop:CONV_VR_RSI}
\vspace{-3mm}
\end{frm-thm}

\vspace{-1mm}
\paragraph{Previous Literature (SDEs for MB-PGF).} \cite{mertikopoulos2018convergence} studied dual averaging using a similar SDE model in the convex setting, under vanishing and persistent volatility. Part of their results are similar, yet less general and not directly comparable. \cite{raginsky2012continuous} studied a specific case of our equations, under constant volatility~(see also~\cite{raginsky2017non} and references therein). \cite{krichene2017acceleration,xu2018continuous, xu2018accelerated} studied extentions to~\cite{mertikopoulos2018convergence} including acceleration~\cite{nesterov2018lectures} and AC-SA~\cite{ghadimi2012optimal}. To the best of our knowledge, there hasn't been yet any analysis of SVRG in continuous-time in the literature.

\subsection{Discrete-time analysis and algebraic equivalence}
\label{sec:non-asy-discrete}
We provide matching algorithmic counterparts (using the same Lyapunov function) for all our non-asymptotic rates in App.~\ref{sec:proofs_ct}, along with Tb.~\ref{tb:summary_rates_GD_main} to summarize the results. We stress that the rates we prove in discrete-time (i.e. for SGD with gradient estimators $\mathcal{G}_\text{MB}$ or $\mathcal{G}_\text{VR}$) hold \textit{without Gaussian noise assumption}. This is a key result of this paper, which indicates that the tools of It\^{o} calculus~\cite{ito1944109} ---which are able to provide more compact proofs~\cite{mertikopoulos2018convergence,raginsky2017non} --- yield calculations which are \textit{equivalent} to the ones used to analyze standard SGD. We invite the curious reader to go through the proofs in the appendix to appreciate this correspondence as well as to inspect Tb.~\ref{tb:summary_rates_GD} in the appendix, which provides a comparison of the discrere-time rates with Thms.~\ref{prop:CONV_GF_smooth},~\ref{prop:CONV_GF_WQC},~\ref{prop:CONV_GF_PL} and~\ref{prop:CONV_VR_RSI}.

\bgroup
\def\arraystretch{2.5}
\begin{table*}[!htbp]
\centering
\begin{tabular}{| l | l | l | l |}
\hline
\textbf{\small{Cond.}} & \textbf{\small{Rate (Discrete-time, no Gaussian assumption)}} &\textbf{\scriptsize{Thm.}}\\
\hline
\scriptsize{($\boldsymbol{\sim}$),\textbf{(H-)},\textbf{(H$\boldsymbol{\sigma}$)}}& \small{$\displaystyle\frac{2 \ (f(x_0)-f(x^\star))}{(h  \varphi_{k+1})}+\frac{h\  d \ L\ \sigma^2_*}{(h \varphi_{k+1})}\sum_{i=0}^k \frac{\psi_i^2}{b_i} h\ $} & \ref{prop:CONV_GD_smooth}\\[5pt]
\hline
\scriptsize{($\boldsymbol{\sim}$),\textbf{(H-)},\textbf{(H$\boldsymbol{\sigma}$)},\textbf{(H\begin{tiny}{WQC}\end{tiny})}}& \small{$\displaystyle\frac{\|x_0-x^\star\|^2}{\tau \ (h \varphi_{k+1}) }+\frac{d \ h \ \sigma^2_*}{\tau \ (h \varphi_{k+1}) }\sum_{i=0}^k \frac{\psi_i^2}{b_i}h\ $}  & \ref{prop:CONV_GD_WQC}\\[5pt]
\hline
\scriptsize{\textbf{(H-)},\textbf{(H$\boldsymbol{\sigma}$)},\textbf{(H\begin{tiny}{WQC}\end{tiny})}}  & \small{$\displaystyle \frac{\|x_0-x^\star\|^2}{2 \ \tau \ (h\varphi_{k+1})} + \frac{h \ d \ \sigma^2_*}{2 \ \tau \ (h\varphi_{k+1})} \sum_{i=0}^{k} (1+\tau \varphi_{i+1}L) \frac{\psi_i^2}{b_i}h\ $}  & \ref{prop:CONV_GD_WQC}\\[5pt]
\hline
\scriptsize{\textbf{(H-)},\textbf{(H$\boldsymbol{\sigma}$)},\textbf{(H\begin{tiny}{P\L{}}\end{tiny})}}& \small{$\displaystyle \prod_{i=0}^{k}(1-\mu \ h\psi_i)(f(x_0)-f(x^\star))+ \frac{h \ d \ L \ \sigma^2_* }{2}\sum_{i=0}^k \frac{\prod_{\ell=0}^{k}(1-\mu \ h\psi_\ell)}{ \prod_{j=0}^{i}(1-\mu \ h\psi_l)} \frac{\psi_i^2}{b_i} h\ $}  &\ref{prop:CONV_GD_PL}\\[5pt]
\hline
\scriptsize{\textbf{(H-)},\textbf{(H\begin{tiny}{RSI}\end{tiny})}}  & \small{$\displaystyle \left(\frac{1+2L^2h^2 m}{h m (\mu-2L^2 h)}\right)^j \|x_0-x^*\|^2\ $ (under variance reduction)}  & \ref{prop:CONV_VR_RSI_discrete}\\[5pt]
\hline
\end{tabular}
\caption{Summary of the rates we show in the appendix for SGD with mini-batch and VR, using a Lyapunov argument inspired by the continuous-time analysis. $(\sim)$ indicates randomized output. The reader should compare the results with Thms.~\ref{prop:CONV_GF_smooth},~\ref{prop:CONV_GF_WQC},~\ref{prop:CONV_GF_PL},~\ref{prop:CONV_VR_RSI} (explicit comparison in the first page of the appendix). For the definition of the quantities in the rates, check App.~\ref{sec:proofs_dt}.}
\label{tb:summary_rates_GD_main}
\end{table*}
\egroup

Now we ask the simple question: \textit{why is this the case?} Using the concept of derivation from abstract algebra, in App.~\ref{sec:algebra} we show that the discrete difference operator and the derivative operator enjoy similar algebraic properties. \textit{Crucially}, this is due to the smoothness of the underlying objective --- which implies a chain-rule\footnote{This is a key formula in the continuous-time analysis to compute the derivative of a Lyapunov function.} for the difference operator. Hence, this equivalence is tightly linked with optimization and might lead to numerous insights. We leave the exploration of this fascinating direction to future work.

\vspace{-1mm}
\paragraph{Literature comparison (algorithms).} Even though partial\footnote{The convergence under weak-quasi-convexity using a learning rate $C/\sqrt{k}$ and a randomized output is studied in~\cite{hardt2016gradient} (Prop. 2.3 under Eq.~2.2 of their paper). On the same line, ~\cite{karimi2016linear} studied the convergence for P\L{}using a learning rate $C/\sqrt{k}$ and assuming bounded stochastic gradients. These results are strictly contained in our rates.} results have been derived for the function classes described above in~\cite{hardt2016gradient,reddi2016proximal}, an in-depth non-asymptotic analysis was still missing. Rates in Tb.~\ref{tb:summary_rates_GD}  (stated above in continuous-time as theorems) provide a generalization to the results of~\cite{moulines2011non} to the weaker function classes we considered (we \textit{never} assume convexity). Regarding SVRG, the rate we report uses a proof similar\footnote{In particular, the lack of convexity causes the factor $L^2$ in the linear rate.} to~\cite{allen2016improved,reddi2016stochastic} and is comparable to ~\cite{johnson2013accelerating}~(under convexity).

%% file: 03_insights.tex
\section{Insights provided by continuous-time models}
\label{sec:insights}

Building on the tools we used so far, we provide novel insights on the dynamics of SGD. First, in order to consider both MB-PGF and VR-PGF at the same time, we introduce a stochastic\footnote{For MB-PGF, $\{\sigma(t)\}_{t\ge0} := \{\sigma(X(t))\}_{t\ge0}$. For VR-PGF, $\{\sigma(t)\}_{t\ge0} := \{\sigma(X(t),X(t-\xi(t)))\}_{t\ge0}$.} matrix process $\{\sigma(t)\}_{t\ge0}$ adapted to the Brownian motion:
\begin{equation}
    \tag{PGF}
    dX(t) = -\psi(t)\nabla f(X(t)) \ dt + \psi(t)\sqrt{h/b(t)} \sigma(t) \ dB(t). 
\end{equation}
We show that annealing the learning rate through a decreasing $\psi(\cdot)$ can be viewed as performing a \textit{time dilation} or, alternatively, as directly \textit{stretching the objective function}. This view is inspired from the use of Girsanov theorem~\cite{girsanov1960transforming} in finance: a deep result in stochastic analysis which is the formal concept underlying the change of measure from real world to "risk-neutral" world.
\subsection{Time stretching through {\O}ksendal's formula}
\label{sec:time_stretch}
We notice that, in Thm.~\ref{prop:CONV_GF_smooth},\ref{prop:CONV_GF_WQC},\ref{prop:CONV_GF_PL}, the time variable $t$ is \textit{always} filtered through the map $\varphi(\cdot)$. Hence, $\varphi(\cdot)$ seems to act as a new time variable. We show this rigorously using {\O}ksendal's time change formula.

\begin{frm-thm}
Let $\{X(t)\}_{t\ge0}$ satisfy PGF and define $\tau(\cdot) = \varphi^{-1}(\cdot)$, where $\varphi(t) = \int_0^t\psi(s)ds$. For all $t\ge 0$, $X\left(\tau(t)\right) = Y(t)$ in distribution, where $\{Y(t)\}_{t\ge0}$ has the stochastic differential
$$dY(t)=-\nabla f(Y(t)) dt + \sqrt{h \ \psi(\tau(t))/b(\tau(t))} \sigma(\tau(t)) \ d B(t).$$
\vspace{-3mm}
\label{thm:time-change}
\end{frm-thm}

\begin{proof}
We use the substitution formula for deterministic integrals combined with {\O}ksendal's formula for time change in stochastic integrals --- a key result in SDE theory. Details in App.~\ref{sec:time-change-proof}.
\end{proof}

\begin{wrapfigure}[15]{r}{0.32\textwidth}
  \centering
    \vspace{-4.5mm}
    \includegraphics[width=0.9\linewidth]{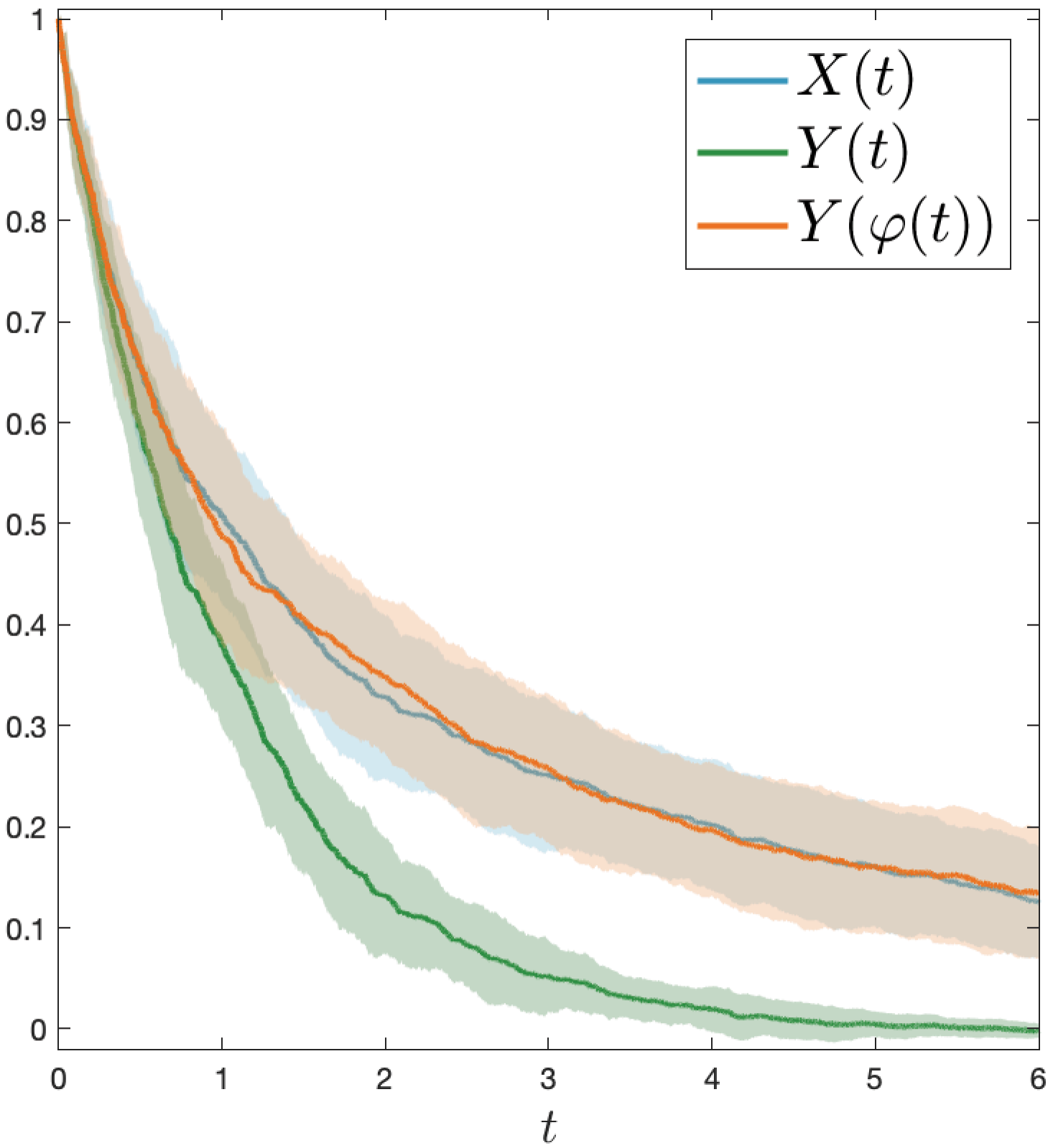}
    \caption{Verification of Thm.~\ref{thm:time-change} on a 1d quadratic (100 samples):  empirically $X(t)\stackrel{d}{=}Y(\varphi(t))$.}
    \label{fig:thm4}
\end{wrapfigure}

\paragraph{Example.} We consider $b(t) = 1$, $\sigma(s) = \sigma \I_d$ and $\psi(t)=1/(t+1)$ (popular annealing procedure~\cite{bottou2018optimization}); we have $\varphi(t) = \log(t+1)$ and $\tau(t) = e^t-1$. $dX(t) = -\frac{1}{t+1}\nabla f(X(t)) dt - \frac{\sqrt{h}\sigma}{t+1}dB(t)$ is s.t. the \textit{sped-up solution} $Y(t)=X(e^{t}-1)$ satisfies
\begin{equation}
dY(t) = -\nabla f(X(t)) dt +\sqrt{h \sigma e^{-t}} dB(t). 
\label{eq:dy_example}
\end{equation}

In the example, Eq.~\eqref{eq:dy_example} is the model for SGD with constant learning rates but rapidly vanishing noise --- which is arguably easier to study compared to the original equation, that also includes time-varying learning rates. Hence, this result draws a connection to SGLD~\cite{raginsky2017non} and to prior work on SDE models~\cite{mertikopoulos2018convergence}, which only considered $\psi(t)=1$. But, most importantly --- Thm.~\ref{thm:time-change} allows for more \textit{flexibility} in the analysis: to derive convergence rates\footnote{The design of the Lyapunov function might be easier if we change time variable. This is the case in our setting, where $\varphi(t)$ comes directly into the Lyapunov functions and would be simply $t$ for the transformed SDE.} one could work with either $X$(as we did in Sec.~\ref{sec:models}) or with $Y$ (and \textit{slow-down} the rates afterwords). 

We verify this result on a one dimensional quadratic, under the choice of parameters in our example, using Euler-Maruyama simulation (i.e. PGD) with $h = 10^{-3}$, $\sigma = 5$. In Fig.~\ref{fig:thm4} we show the mean and standard deviation relative to 20 realization of the Gaussian noise.

Note that in the case of variance reduction, the volatility is decreasing as a function of time~\cite{allen2016variance}, even with $\psi(t)=1$. Hence one gets a similar result without the change of variable.

\subsection{Landspace stretching via solution feedback}
\label{sec:land-stretch}

Consider the (potentially non-convex) quadratic $f(x) = \langle x-x^\star, H(x-x^\star)\rangle$. WLOG we assume $x^\star = 0_d$ and that $H$ is diagonal. For simplicity, consider again the case $b(t)=1$, $\sigma(s) = \sigma \I_d$ and $\psi(t)=1/(t+1)$. PGF reduces to a linear stochastic system:
\vspace{-1mm}
 \begin{wrapfigure}[14]{r}{0.32\textwidth}
  \vspace{1.5mm}
  \centering
    \includegraphics[width=0.9\linewidth]{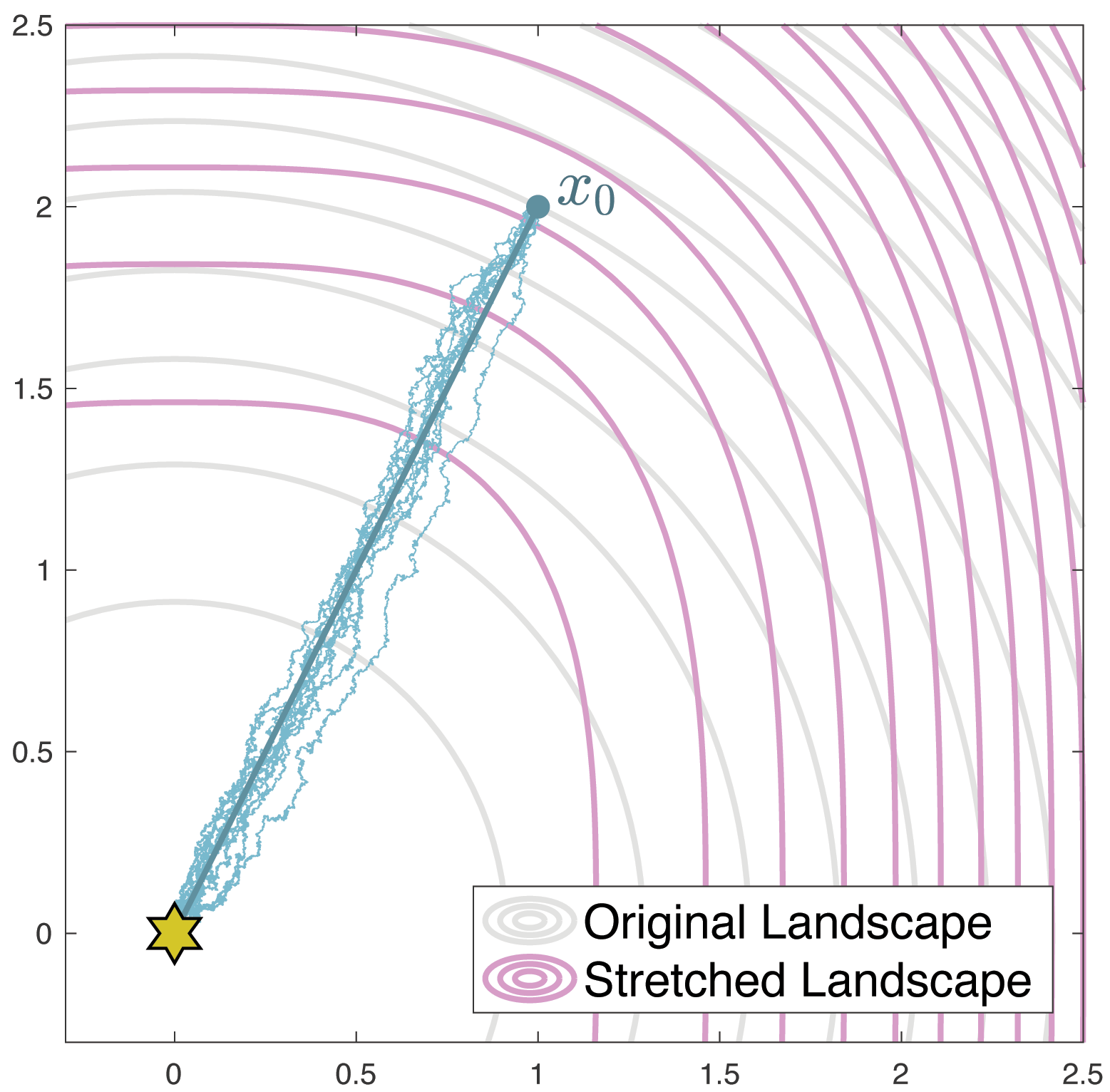}
    \caption{Landscape stretching for an isotropic paraboloid.}
    \label{fig:stretch}
\end{wrapfigure}
$$dX(t) = -\frac{1}{t+1} H X(t)dt + \frac{h \sigma}{t+1}dB(t).$$ 
By the variation-of-constants formula~\cite{mao2007stochastic}, the expectation evolves without bias:  $d\E [X(t)] = -\frac{1}{t+1} H \E [X(t)] dt$. If we denote by $u^i(t)$ the $i$-th coordinate of $\E[X(t)]$ we have $\frac{d}{dt} u^i(t) = -\frac{\lambda_i}{t+1}u^i(t)$, where $\lambda_i$ is the eigenvalue relative to the $i$-th direction. Using separation of variables, we find $u^i(t) =(t+1)^{-\lambda_i} u^i_0$. Moreover, we can invert space and time: $t = \left(u^i_0/u^i(t)\right)^{1/\lambda_i}-1$. \textit{Feeding back} this equation into the original differential --- the system becomes autonomous: $$\frac{d}{dt} u^i(t) = -\lambda_i (u^i_0)^{-\frac{1}{\lambda_i}} u^i(t)^{1+\frac{1}{\lambda_i}}.$$

From this simple derivation we get two important insights on the dynamics of PGF:
\begin{enumerate}[leftmargin=*]
    \setlength\itemsep{0em}
    \item Comparing the solution $u^i(t) =(t+1)^{-\lambda_i} u^i_0$ with the solution one would obtain with $\psi(t)=1$, that is $e^{-\lambda_i t} u^i_0$ --- we notice that the dynamics in the first case is much slower: we get polynomial convergence and divergence (when $\lambda_i\le0$) as opposed to exponential. This quantitatively shows that decreasing the learning rate could slow down (from exponential to polynomial) the dynamics of SGD around saddle points. However, note that, even though the speed is different, $u^i(\cdot)$ and $v^i(\cdot)$ move along the same path\footnote{One is the time-changed version of the other (consider Thm.~\ref{thm:time-change} with $\sigma(t) = 0$), see also Fig.~\ref{fig:thm4}.} by Thm~\ref{thm:time-change}.
    \item Inspecting the equivalent formulation $\frac{d}{dt} u^i(t) = -\lambda_i (u^i_0)^{-\frac{1}{\lambda_i}} u^i(t)^{1+\frac{1}{\lambda_i}}$, we notice with surprise --- that this is a gradient system. Indeed the RHS can be written as $C(\lambda_i,u^i_0)\nabla g_i(u^i(t))$, where $g_i(x) = x^{2+\frac{1}{\lambda_i}}$ is the \textit{equivalent landscape} in the $i$-th direction. In particular, PGF on the simple quadratic $\frac{1}{2}\|x\|^2$ with learning rate decreasing as $1/t$ \textit{behaves in expectation like PGF with constant learning rate on a cubic}. This shines new light on the fact that, as it is well known from the literature~\cite{nemirovsky1983problem}, by decreasing the learning rate we can only achieve sublinear convergence rates on strongly convex stochastic problems. From our perspective, this happens simply because the equivalent stretched landscape has vanishing curvature --- hence, it is not strongly convex.  We illustrate this last example in Fig.~\ref{fig:stretch} and note that the stretching effect is tangent to the expected solution (in solid line).
\end{enumerate}

 We believe the landscape stretching phenomenon we just outlined to be quite general and to also hold asymptotically under strong convexity\footnote{Perhaps also in the neighborhood of any hyperbolic fixed point, with implications about saddle point evasion.}: indeed it is well known that, by Taylor's theorem, in a neighborhood of the solution to a strongly convex problem the cost \textit{behaves as its quadratic approximation}. In dynamical systems, this linearization argument can be made precise and goes under the name of Hartman-Grobman theorem (see e.g.~\cite{perko2013differential}). Since the SDE we studied is memoryless (no momentum), at some point it will necessarily enter a neighborhood of the solution where the dynamics is described by result in this section. We leave the verification and formalization of the argument we just outlined to future research.

%% file: 05_conclusion.tex
% !TEX root = aistats18_svrg.tex

% \section{Discussion}

% In this work, we studied a model that integrates both mini-batch and variance-reduced methods. We derived rates of convergences for several classes of non-convex functions using a simple Lyapunov analysis. We note that convergence rates for discrete-time mini-batch algorithms had previously been derived in ~\cite{hardt2016gradient} and~\cite{karimi2016linear} for WQC and P\L{} functions respectively, although the analysis in Section~\ref{sec:MB} applies to a much broader set of step sizes (instead of a fixed step size in prior work). To the best of our knowledge, the analysis of these methods in continuous-time is novel.
% We also derived a novel continuous-time perspective for variance-reduction, from which we derived a linear rate of convergence for the discrete-time algorithm SVRG without having to rely on convexity.

% Finally, we note that SDEs have already delivered new insights for the analysis of mirror-descent or accelerated gradient-descent~\cite{krichene2015accelerated, su2014differential}, to which our approach could be combined in order to analyze an accelerated variance-reduced method.

%the convergence rate obtained in the continuous case is more optimistic than the one derived in the discrete case \highlight{Still the case right?}. This indicates that more accurate discretization techniques could perhaps reach faster convergence rates for variance-reduced techniques for the non-convex case~\cite{reddi2016stochastic}.

\section{Conclusion}

We provided a detailed comparisons and analysis of continuous- and discrete-time methods in the context of stochastic non-convex optimization. Notably, our analysis covers the variance-reduced method introduced in~\cite{johnson2013accelerating}. The continuous-time perspective allowed us to deliver new insights about how decreasing step-sizes lead to time and landscape stretching. 
There are many potential interesting directions for future research such as extending our analysis to mirror-descent or accelerated gradient-descent~\cite{krichene2015accelerated, su2014differential}, or to study state-of-the-art stochastic non-convex optimizers such as Natasha~\cite{allen2017natasha}. Finally, we believe it would be interesting to expand the work of~\cite{li2015stochastic,li2017convergence} to better characterize the convergence of MB-SGD and SVRG to the SDEs we studied here, perhaps with some asymptotic arguments similar to the ones used in mean-field theory~\cite{benaim1998recursive,benaim2008class}.   

\section*{Acknowledgements}
The first author would like to thank Enea Monzio Compagnoni for his proof of Theorem C.1 and Thomas Hofmann for his valuable comments on the first version of this manuscript.

%% file: appendix.tex
% !TEX root = aistats18_svrg.tex

\newpage
\appendix
\onecolumn
{\Huge \textbf{Appendix}}
\section{Summary of the rates derived in this paper}
\begin{chapquote}{E.T. Bell, \textit{Men of Mathematics}, 1937}
"A major task of mathematics today is to harmonize the continuous and the discrete, to include them in one comprehensive mathematics, and to eliminate obscurity from both."
\end{chapquote}
\vspace{-5mm}
\bgroup
\def\arraystretch{2.5}
\begin{table*}[!htbp]
\centering
\begin{tabular}{| l | l | l | l |}
\hline
\textbf{\small{Cond.}} & \textbf{\small{Rate (Continuous-time)}} &\textbf{\scriptsize{Thm.}}\\
\hline
\scriptsize{($\boldsymbol{\sim}$),\textbf{(H-)},\textbf{(H$\boldsymbol{\sigma}$)}}& \small{$\displaystyle\frac{f(x_0)-f(x^{\star})}{\varphi(t)}+ \frac{h \ d \ L \ \sigma_*^2}{2 \ \varphi(t)}\int_0^t\frac{\psi(s)^2}{b(s)}ds
\ $} & \ref{prop:CONV_GF_smooth}\\[5pt]
\hline
\scriptsize{($\boldsymbol{\sim}$),\textbf{(H-)},\textbf{(H$\boldsymbol{\sigma}$)},\textbf{(H\begin{tiny}{WQC}\end{tiny})}}& \small{$\displaystyle\frac{\|x_0-x^{\star}\|^2}{2 \ \tau \ \varphi(t) }+ \frac{h \ d \ \sigma^2_*}{2 \ \tau \ \varphi(t)}\int_0^t\frac{\psi(s)^2}{b(s)}ds\ $}  & \ref{prop:CONV_GF_WQC}\\[5pt]
\hline
\scriptsize{\textbf{(H-)},\textbf{(H$\boldsymbol{\sigma}$)},\textbf{(H\begin{tiny}{WQC}\end{tiny})}}  & \small{$\displaystyle \frac{\|x_0-x^{\star}\|^2}{2 \ \tau \ \varphi(t) }+ \frac{h \ d \ \sigma^2_*}{2 \ \tau \ \varphi(t) }\int_0^t(L \ \tau \ \varphi(s)+1)\frac{\psi(s)^2}{b(s)}ds\ $}  & \ref{prop:CONV_GF_WQC}\\[5pt]
\hline
\scriptsize{\textbf{(H-)},\textbf{(H$\boldsymbol{\sigma}$)},\textbf{(H\begin{tiny}{P\L{}}\end{tiny})}}& \small{$\displaystyle e^{-2\mu \varphi(t)}(f(x_0)-f(x^{\star})) + \frac{h \ d \ L \ \sigma^2_* }{2} \int_0^t\frac{\psi(s)^2}{b(s)} e^{-2\mu (\varphi(t)-\varphi(s))}ds\ $}  &\ref{prop:CONV_GF_PL}\\[5pt]
\hline
\scriptsize{\textbf{(H-)},\textbf{(H\begin{tiny}{RSI}\end{tiny})}}  & \small{$\displaystyle \left(\frac{1+2 h L^2\T}{\T(\mu-2hL^2)}\right)^j \|x_0-x^*\|^2\ $ (\underline{with variance reduction})}  & \ref{prop:CONV_VR_RSI}\\[5pt]
\hline
\hline
\hline
\textbf{\small{Cond.}} & \textbf{\small{Rate (Discrete-time, no Gaussian assumption)}} &\textbf{\scriptsize{Thm.}}\\
\hline
\scriptsize{($\boldsymbol{\sim}$),\textbf{(H-)},\textbf{(H$\boldsymbol{\sigma}$)}}& \small{$\displaystyle\frac{2 \ (f(x_0)-f(x^\star))}{(h  \varphi_{k+1})}+\frac{h\  d \ L\ \sigma^2_*}{(h \varphi_{k+1})}\sum_{i=0}^k \frac{\psi_i^2}{b_i} h\ $} & \ref{prop:CONV_GD_smooth}\\[5pt]
\hline
\scriptsize{($\boldsymbol{\sim}$),\textbf{(H-)},\textbf{(H$\boldsymbol{\sigma}$)},\textbf{(H\begin{tiny}{WQC}\end{tiny})}}& \small{$\displaystyle\frac{\|x_0-x^\star\|^2}{\tau \ (h \varphi_{k+1}) }+\frac{d \ h \ \sigma^2_*}{\tau \ (h \varphi_{k+1}) }\sum_{i=0}^k \frac{\psi_i^2}{b_i}h\ $}  & \ref{prop:CONV_GD_WQC}\\[5pt]
\hline
\scriptsize{\textbf{(H-)},\textbf{(H$\boldsymbol{\sigma}$)},\textbf{(H\begin{tiny}{WQC}\end{tiny})}}  & \small{$\displaystyle \frac{\|x_0-x^\star\|^2}{2 \ \tau \ (h\varphi_{k+1})} + \frac{h \ d \ \sigma^2_*}{2 \ \tau \ (h\varphi_{k+1})} \sum_{i=0}^{k} (1+\tau \varphi_{i+1}L) \frac{\psi_i^2}{b_i}h\ $}  & \ref{prop:CONV_GD_WQC}\\[5pt]
\hline
\scriptsize{\textbf{(H-)},\textbf{(H$\boldsymbol{\sigma}$)},\textbf{(H\begin{tiny}{P\L{}}\end{tiny})}}& \small{$\displaystyle \prod_{i=0}^{k}(1-\mu \ h\psi_i)(f(x_0)-f(x^\star))+ \frac{h \ d \ L \ \sigma^2_* }{2}\sum_{i=0}^k \frac{\prod_{\ell=0}^{k}(1-\mu \ h\psi_\ell)}{ \prod_{j=0}^{i}(1-\mu \ h\psi_l)} \frac{\psi_i^2}{b_i} h\ $}  &\ref{prop:CONV_GD_PL}\\[5pt]
\hline
\scriptsize{\textbf{(H-)},\textbf{(H\begin{tiny}{RSI}\end{tiny})}}  & \small{$\displaystyle \left(\frac{1+2L^2h^2 m}{h m (\mu-2L^2 h)}\right)^j \|x_0-x^*\|^2\ $ (\underline{with variance reduction})}  & \ref{prop:CONV_VR_RSI_discrete}\\[5pt]
\hline
\end{tabular}
\caption{Summary of the main convergence results for MB-PGF and VR-PGF compared to SGD with mini-batch or variance reduced gradient estimators. $(\sim)$ indicates a randomized output. For the definition of the quantities in the rates, check App.~\ref{sec:proofs_ct} and App.~\ref{sec:proofs_dt}.}
\label{tb:summary_rates_GD}
\end{table*}
\vspace{-2mm}
\subsection{Correspondences between continuous and discrete-time}
\vspace{-2mm}
We note the following simple correspondences:
\vspace{-2mm}
\begin{enumerate}[leftmargin=*]
    \item $h$ corresponds to $dt$. The rates are \textit{not} simplified to highlight the equivalence.
    \item $h\varphi_{k+1}$ corresponds to $\varphi(t)$. Indeed, $\varphi(t) = \int_0^t \psi(s) ds\simeq\sum_{i=0}^{k} \psi_k h = \varphi_{k+1} h$.
    \item The same argument holds for the exponential and the power, since $e^{a t}\simeq (1+ah)^k$.
    \item The rates for variance reduction match since by definition $\T = m \ h$.
    \item The difference only comes into a few constants which do not depend on the parameters of the problem nor on the algorithm. Those differences are due to higher order terms in the algorithm.
\end{enumerate}
\egroup

\subsection{Algebraic equivalence}
\label{sec:algebra}
In this section we motivate the equivalence outlined in Tb.~\ref{tb:summary_rates_GD} in the deterministic setting, although a similar derivation can easily be performed in the stochastic setting using the diffusion operator instead of the derivative (we introduce this object in App.~\ref{sec:stoc_cal}). We take inspiration from a concept in abstract algebra and we combine it with smoothness --- a common assumption in optimization.

\begin{definition}
Let $A$ be an algebra over a field $F$. A \textit{derivation} is a linear map $D:A\to A$ that satisfies Leibniz's law: $D(ab) = aD(b)+D(a)b$.
\end{definition}

Consider the vector space of $d$-dimensional sequences over $\N$ equipped with pointwise and elementwise product and sum, which we denote as $\R^{d\times\infty}$; this is trivially an algebra. Next, let us define the sequence $D_h(x)$ (still in the algebra) pointwise: for all $k\in\N$ $$\left[D_h(x)\right]_k =: D_h(x,k) = \frac{x_{k+1}-x_k}{h}.$$
 
 Notice that GD can be written as $D_h(x,k) = -\nabla f(x_k)$, which resembles the gradient flow equation $\frac{d}{dt} X(t) = -\nabla f(X(t))$. The \textit{crucial question} is whether the continuous time derivative $\frac{d}{dt}$ and the operator $D_h$ have the same properties. This would motivate an algebraic equivalence between continuous and discrete time in optimization.

To start, we show that $D_h$ is \textit{almost} a derivation. We denote by $x^+$ the one-step-ahead $x$ sequence: $x^+_k = x_{k+1}$ for all $k\in\N$.

\begin{enumerate}[leftmargin=*]
    \item Let $x,y\in \R^{d\times\infty}$ and $k\in\N$, $D_h(x+y,k) = D_h(x,k)+D_h(y,k).$
    \item Let $x\in \R^{d\times\infty}$ , $a\in\R$ and $k\in\N$, $D_h(ax,k) = a D_h(x,k)$.
    \item Let $x,y\in \R^{d\times\infty}$; for all $k\in\N$, 
    \begin{multline*}
        D_h(xy,k) = \frac{1}{h}(y_{k+1}x_{k+1}-y_k x_k) \\ =\frac{1}{h}((y_{k+1}-y_k)x_{k+1} + y_k x_{k+1}-y_k x_k) = \frac{y_{k+1}-y_k}{h} x_{k+1} + y_k \frac{x_{k+1}-x_k}{h}.
    \end{multline*}
    Therefore $D(xy) = x^+ D_h(y) + D_h(x) y $.
\end{enumerate}

Since we will only care about the value of $D_h(x)$ at iteration $k$, we are going to deal with the pointwise map $D_h(x,k)$ and deviate from the algebraic definition.

For a complete correspondence to continuous time, we still need a chain rule. For this, we need a bit more flexibility in the definition of $D_h$: let $g:\R^d\to\R$ be $L$-smooth, we define
$$D_h(g\circ x,k):=\frac{g(x_{k+1})-g(x_{k})}{h}.$$
Smoothness gives us a chain rule: we have 
$$g(x_{k+1})\le g(x_{k})+ \langle\nabla g(x_k), x_{k+1}-x_k \rangle + \frac{L}{2}\|x_{k+1}-x_k\|^2;$$
hence
$$D_h(g\circ x,k) \le \langle\nabla g(x_k), \frac{x_{k+1}-x_k}{h} \rangle + \frac{L}{2h}\|x_{k+1}-x_k\|^2 = \langle\nabla g(x_k), D_h(x,k) \rangle + \frac{L h}{2}\|D_h(x,k)\|^2.$$

We condense our findings in the box below
\begin{tcolorbox}
Let $\{x_k\}_{k\in\N}$ and $\{y_k\}_{k\in\N}$ be sequences of $\R^d$ vectors and let $g:\R^d\to\R$ be $L$-smooth. 
\begin{itemize}[leftmargin=*]
    \item Linearity : $ D_h(x+y,k) = D_h(x,k)+D_h(y,k)$, $a\in\R$ and $D_h(ax,k) = aD_h(x,k)$.
    \item Product rule: $D_h(xy,k) = D_h(x,k) y_{k} + x_{k+1} D_h(y,k)$.
    \item Chain rule: $D_h(g(x),k) \le \langle\nabla g(x_k), D_h(x,k) \rangle + \frac{L h}{2}\|D_h(x,k)\|^2$.
\end{itemize}
\end{tcolorbox}

This shows that the operations in continuous time and in discrete time are algebraically very similar, motivating the success behind the matching rates summarized in Tb.~\ref{tb:summary_rates_GD}. Indeed, taking $h\to 0$ we recover the normal derivation rules from calculus.

\section{Stochastic Calculus}
\label{sec:stoc_cal}
In this appendix we summarize some important results in the analysis of Stochastic Differential equations~\cite{mao2007stochastic,oksendal1990stochastic}. The notation and the results in this section will be used extensively in all proofs in this paper. We assume the reader to have some familiarity with Brownian motion and with the definition of stochastic integral (Ch.~1.4 and 1.5 in~\cite{mao2007stochastic}).

\subsection{It\^{o}'s lemma and Dynkin's formula}
\label{subsec:ito}

We start with some notation: let $\left(\Omega, \mathcal{F}, \{\mathcal{F}(t)\}_{t\ge0}, \mathbb{P}\right)$ be a filtered probability space. We say that an event $E\in\mathcal{F}$ holds almost surely (a.s.) in this space if $\mathbb{P}(E)=1$. We call $\mathcal{L}^p([a,b],\R^d)$, with $p>0$, the family of $\R^d$-valued $\mathcal{F}(t)$-adapted processes $\{f(t)\}_{a\le t\le b}$ such that $$\int_a^b\|f(t)\|^p dt\le\infty.$$ Moreover, we denote by   $\mathcal{M}^p([a,b],\R^d)$, with $p>0$, the family of $\R^d$-valued processes $\{f(t)\}_{a\le t\le b}$ in $\mathcal{L}([a,b],\R^d)$ such that $\E\left[\int_a^b\|f(t)\|^p dt\right]\le\infty$. 
We will write $h \in \mathcal{L}^p\left(\R_{+},\R^d\right)$, with $p>0$, if $h \in \mathcal{L}^p\left([0,T],\R^d\right)$ for every $T>0$. Same definitions holds for matrix
valued functions using the Frobenius norm $\|A\| := \sqrt{\sum_{ij}|A_{ij}|^2}$.

Let $B=\{B(t)\}_{t\ge0}$ be a one dimensional Brownian motion defined on our probability space and let $X=\{X(t)\}_{t\ge0}$ be an $\mathcal{F}(t)$-adapted process taking values on $\R^d$.

\begin{definition}
Let $b \in \mathcal{L}^1\left(\R_{+},\R^d\right)$ (the drift) and $\sigma \in \mathcal{L}^2\left(\R_{+},\R^{d\times m}\right)$ (the volatility). $X$ is an It\^{o} process if it takes the form
\vspace{-3mm}
$$X(t) = x_0 + \int_0^t f(s)ds + \int_0^t \sigma(s) dB(s).$$
We shall say that $X$ has the stochastic differential
\begin{equation}
    dX(t) = f(t) dt + \sigma(t) dB(t).
    \label{eq:stoc_diff}
\end{equation}
\label{def:SDE_ito_process}
\end{definition}
\vspace{-5mm}
In this paper we indicate as $\partial_{x} f(x,t)$ the $d$-dimensional vector of partial derivatives of a scalar function $f:\R^d\times [0,\infty)\to \R$ w.r.t. each component of $x$. Moreover, we call $\partial_{xx} f(x,t)$ the $d\times d$-matrix of partial derivatives of each component of $\partial_{x} f(x,t)$ w.r.t each component of $x$. We now state the celebrated \textbf{It\^{o}'s lemma}.
\vspace{2mm}
\begin{frm-thm-app}[It\^{o}'s lemma]
Let $X$ be an It\^{o} process with stochastic differential $dX(t) = f(t) dt + \sigma(t) dB(t)$. Let $\EE\left(x,t\right)$ be twice continuously differentiable in $x$ and continuously differentiable in $t$, taking values in $\R$. Then $\EE(X(t),t)$ is again an It\^{o} process with stochastic differential
\begin{multline}
        d\EE(X(t),t) = \partial_t \EE (X(t), t)) dt + \langle\partial_{x} \EE(X(t),t), f(t)\rangle dt\\+  \frac{1}{2}\tr\left(\sigma(t)\sigma(t)^T\partial_{xx}\EE(X(t),t) \  \right)dt + \langle\partial_{x}\EE(x(t),t),\sigma(t)\rangle dB(t),
\end{multline}
which we sometimes write as
$$d\EE = \partial_t \EE dt + \langle\partial_{x} \EE, dX\rangle + \frac{1}{2}\tr\left( \sigma\sigma^T\partial_{xx}\EE \right)dt$$

\label{lemma:SDE_ito}
\end{frm-thm-app}

Following~\cite{mao2007stochastic}, we introduce the \textbf{It\^{o} diffusion differential operator} $\mathscr{A}$:
\begin{equation}
\mathscr{A}(\cdot)= \partial_t(\cdot) + \langle\partial_{x} (\cdot), b(t)\rangle + \frac{1}{2}\tr\left(\sigma(t)\sigma(t)^T\partial_{xx}(\cdot)\right).
    \label{eqn:SDE_diffusion_op}
\end{equation}

It is then clear that, thanks to It\^{o}'s lemma,
$$d\EE(X(t),t) = \mathscr{A}\EE(X(t),t) dt + \langle\EE_X(X(t),t),\sigma(t)dB(t)\rangle .$$

Moreover, by the definition of an It\^{o} process, we know that at any time $t>0$,
$$\EE(X(t),t) = \EE(x_0,0) + \int_0^t \mathscr{A}\EE(X(s),s) ds + \int_0^t \langle\partial_{x}\EE(X(s),s),\sigma(s)dB(s)\rangle . \quad \quad a.s.$$
Taking the expectation the stochastic integral vanishes \footnote{Because $\langle\partial_{x}\EE(X(t),t),\sigma(t)\rangle\in \mathcal{M}^2([0,T],\R)$, see e.g. Thm.~1.5.8~\cite{mao2007stochastic}} and we have
\begin{equation}
    \E[\EE(X(t),t)]-\EE(x_0,0) = \E \left[\int_0^t \mathscr{A}\EE(X(t),t) dt\right].
    \label{eq:dynkin}
\end{equation}
This result can be generalized for stopping times and is known as \textbf{Dynkin's formula}.

\subsection{Stochastic differential equations}
\label{subsec:SDE}
Stochastic Differential Equations (SDEs) are equations of the form
\begin{equation*}
    dX = b(X,t)dt + \sigma(X,t) dB(t).
\end{equation*}

Notice that this equation is different from Eq.~\eqref{eq:stoc_diff}, since $X$ also appears on the RHS. Hence, we need to define what it means for a stochastic process $X=\{X(t)\}_{t\ge0}$ with values in $\R^d$ to solve an SDE.

\begin{definition}
Let $X$ be as above with deterministic initial condition $X(0) = x_0$. Assume $b:\R^d\times[0,T]\to\R^d$ and $\sigma:\R^d\times[0,T]\to\R^{d\times m}$ are Borel measurable; $X$ is called a solution to the corresponding SDE if
\begin{enumerate}[leftmargin=*]
    \item $X$ is continuous and $\mathcal{F}(t)$-adapted;
    \item $b \in \mathcal{L}^1\left([0,T],\R^d\right)$;
    \item $\sigma \in \mathcal{L}^2\left([0,T],\R^{d\times m}\right)$;
    \item For every $t \in [0,T]$
    $$X(t) = x_0 + \int_0^t b(X(s),s)ds + \int_0^t\sigma(X(s),s) dB(s) \ \ \ a.s.$$
\end{enumerate}
Moreover, the solution $X(t)$ is said to be unique if any other solution $X^\star(t)$ is such that
$$\mathbb{P}\left\{X(t) = X^\star(t), \text{ for all } 0\le t\le T\right\}=1.$$
\label{def:SDE_sol}
\end{definition}
\vspace{-5mm}
Notice that the solution to a SDE is an It\^{o} process; hence we can use It\^{o}'s Formula (Thm.~\ref{lemma:SDE_ito}). The following theorem gives a sufficient condition on $b$ and $\sigma$ for the existence of a solution to the corresponding SDE.

\begin{frm-thm-app}
Assume that there exist two positive constants $\bar K$ and $K$ such that
\begin{enumerate}
    \item (Global Lipschitz condition) for all $x,y\in \R^d$ and $t\in [0,T]$
    $$\max \{\|b(x,t) - b(y,t)\|,  \ \|\sigma(x,t)-\sigma(y,t)\|\} \le \bar K \|x-y\|^2;$$
    \item (Linear growth condition) for all $x\in \R^d$ and $t\in [0,T]$
    $$\max\{\|b(x,t)\|, \ \|\sigma(x,t)\|\}\le K(1+ \|x\|).$$
\end{enumerate}
Then, there exists a unique solution $X$ to the corresponding SDE , and $X\in \mathcal{M}^2([0,T],\R^d).$
\label{thm:SDE_existence_uniqueness_global}
\end{frm-thm-app}

\paragraph{Numerical approximation.} Often, SDEs are solved numerically. The simplest algorithm to provide a sample path $(\hat x_k)_{k\ge 0}$ for $X$, so that $X(k\Delta t) \approxeq x_k$ for some small $\Delta t$ and for all $k\Delta t\le M$, is called Euler-Maruyama (Algorithm~\ref{algo:EulerMaruryama_SDE}). For more details on this integration method and its approximation properties, the reader can check~\cite{mao2007stochastic}.
\begin{algorithm}
\caption{Euler-Maruyama integration method for a SDE}
    \label{algo:EulerMaruryama_SDE}
\begin{algorithmic}
    \INPUT{The drift $b$ and the volatility $\sigma$; the initial condition $x_0$}
    \STATE fix a stepsize $\Delta t$;
    \STATE initialize $\hat x_0 = x_0$;
    \STATE $k=0$;
    \WHILE{$k \le \left\lfloor\frac{T}{\Delta t}\right\rfloor$}
        \STATE sample some $d$-dimensional Gaussian noise $Z_k\sim\mathcal{N}(0,\I_d)$;
        \STATE compute $\hat x_{k+1} = \hat x_k + \Delta t \ b(\hat x_k,k \Delta t)+ \sqrt{\Delta t} \ \sigma(\hat x_k,k \Delta t) Z_k;$
        \STATE $k=k+1$;
    \ENDWHILE
    \OUTPUT the approximated sample path $(\hat x_k)_{0\le k\le\left\lfloor\frac{T}{\Delta t}\right\rfloor }$
\end{algorithmic}
\end{algorithm}

\subsection{Functional SDEs}
\label{subsec:SDDE}

SDEs describe Markovian (also know as memoryless) processes: a Markovian process is a system where the current state completely determines the future evolution. Indeed, in an SDE, the RHS only depends on $X(t)$ and on $t$. To model variance-reduction methods such as SVRG~\cite{johnson2013accelerating}, we will need a continuous time model which also retains some information about the past. This was also noted in~\cite{he2018differential}.

First, we introduce Functional Stochastic Differential Equations (FSDEs) which are equations of the form
$$dX = b(X_{(0,t]},t)dt + \sigma(X_{(0,t]},t) dB(t),$$
where $X_{(0,t]}$ is the past history of $X$ up to time $t$. Here we focus on a particular type of FSDE, namely Stochastic Differential Delay Equations (SDDEs):

\begin{equation*}
 dX(t) = b(X(t),X(t-\xi(t)),t)dt + \sigma(X(t), X(t-\xi(t)),t) dB(t), 
\end{equation*}

where $\xi(t)\in[0,\tau]$ is the delay at time $t$. As we did in the last subsection for SDEs, we need to define what it means for a stochastic process $X=\{X(t)\}_{t\ge-\tau}$ with values in $\R^d$ to solve an SDDE

\begin{definition}
Let $X$ be as above with deterministic initial condition $X(s) = x_0 \;$ for $-\tau\le s\le 0$. Assume $b:\R^d\times\R^d\times[0,T]\to\R^d$, $\xi:\R_+\to[0,\tau]$ and $\sigma:\R^d\times\R^d\times[0,T]\to\R^{d\times m}$ are Borel measurable; $X$ is called a solution to the corresponding SDDE if

\begin{enumerate}[leftmargin=*]
    \item $X$ is continuous and $\mathcal{F}(t)$-adapted;
    \item $b \in \mathcal{L}^1\left([0,T],\R^d\right)$;
    \item $\sigma\in\mathcal{L}^2\left([0,T],\R^{d\times m}\right)$;
    \item For every $t \in [0,T]$
    $$X(t) = x_0 + \int_0^t b(X(s),X(s-\xi(s)),s)ds + \int_0^t\sigma(X(s), X(s-\xi(s)),s) dB(s) \ \ \ a.s.$$
\end{enumerate}
Moreover, a solution $X(t)$ is said to be unique if any other solution $X^\star(t)$ is such that
$$\mathbb{P}\left\{X(t) = X^\star(t), \text{ for all } -\tau\le t\le T\right\}=1.$$
\label{def:SDDE_sol}
\end{definition}
We state now one existence and uniqueness theorem for SDDEs, which is adapted from equations 5.2 and 5.3 in~\cite{mao2007stochastic}.
\newpage

\begin{frm-thm-app}
Assume that there exist two positive constants $\bar K$ and $K$ such that for all $x,\bar x, y, \bar y \in \R^d$ and for all $t\in [0,T]$
\begin{enumerate}
    \item (Lipschitz condition) 
    $$\max \{\|b(x,y,t) - b(\bar x,\bar y,t)\|,  \ \|\sigma(x,y,t)-\sigma(\bar x,\bar y,t)\|\} \le \bar K (\|x-\bar x\|+ \|y-\bar y\|);$$
    \item (Linear growth condition)
    $$\max\{\|b(x,y,t)\|, \ \|\sigma(x,y,t)\|\}\le K(1+ \|x\|+\|y\|).$$
\end{enumerate}
Then there exists a unique solution $X$ to the corresponding SDDE and $X\in\mathcal{M}^2([-\tau, T],\R^d).$
\label{thm:SDE_existence_uniqueness_SDDE}
\end{frm-thm-app}

\paragraph{Numerical approximation.} Often, SDSEs are solved numerically. Algorithm~\ref{algo:EulerMaruryama_SDE} can easily be modified to work with SDDEs (see (Algorithm~\ref{algo:EulerMaruryama_SDDE})). For more details on approximation error SDDEs, we refer the reader to Chapter 5 in~\cite{mao2007stochastic}.
\begin{algorithm}
\caption{Euler-Maruyama integration method for a SDDE}
    \label{algo:EulerMaruryama_SDDE}
\begin{algorithmic}
    \INPUT{The drift $b$ and the volatility $\sigma$; the initial condition $x_0$}
    \STATE fix a stepsize $\Delta t$
    \STATE compute $q= \left\lfloor\frac{\tau}{\Delta t}\right\rfloor$;
    \STATE initialize $\hat x_k = x_0$ for $-q \le k \le 0$;
    \STATE $k=0$;
    \WHILE{$k \le \left\lfloor\frac{T}{\Delta t}\right\rfloor$}
        \STATE sample some $d$-dimensional Gaussian noise $Z_k\sim\mathcal{N}(0,\I_d)$;
        \STATE compute $\hat x_{k+1} = \hat x_k + \Delta t \ b(\hat x_k,\hat x_{k-q}, k \Delta t)+ \sqrt{\Delta t} \ \sigma(\hat x_k,\hat x_{k-q}, k \Delta t) Z_k;$
        \STATE $k=k+1$;
    \ENDWHILE
    \OUTPUT the approximated sample path $(\hat x_k)_{-q\le k\le\left\lfloor\frac{T}{\Delta t}\right\rfloor }$
\end{algorithmic}
\end{algorithm}

\subsection{Time change in stochastic analysis}

We conclude this appendix with a useful formula from~\cite{oksendal2003stochastic}, which is the equivalent to a chain rule for stochastic processes. We use this formula in Sec.~\ref{sec:time_stretch}.

\begin{frm-thm-app}[Time change formula for It\^{o} integrals]
Let $c:\R_+\to\R_+$ be a strictly positive continuous function and $\beta(t) = \int_0^t c(s)ds$. Denote by $\alpha(\cdot)$ the inverse of $\beta(\cdot)$ and suppose it is continuous. Let $\{B(t)\}_{t\ge0}$ be an $m$-dimensional Brownian Motion and let the stochastic process $\{v(s)\}_{s\ge 0}$ with $v(s)\in\R^{n\times m}$ be Borel measurable in time, adapted to the natural filtration of $B$ and $\mathcal{M}^2(\R_+,\R^d)$. Define 
$$\tilde B(t) = \int_0^t \sqrt{c(s)}dB(s).$$
Then $\{\tilde B(t)\}_{t\ge0}$ is a Brownian Motion and we have
$$\int_0^{\alpha(t)}v(s)dB(s) = \int_0^t \sqrt{\alpha'(s)}v(\alpha(s)) d\tilde B(s), \ \ \ a.s.$$
\label{thm:oxsendall}
\end{frm-thm-app}

%%%%%%%%%%%%%%%%%%%%%%%%%%%%%%%%%%%%%%%%%%%%%%%%%%%%%%%%%%%%%%%%%%%%%%%%%%%%%%%%
%%%%%%%%%%%%%%%%%%%%%%%%%%%%%%%%%%%%%%%%%%%%%%%%%%%%%%%%%%%%%%%%%%%%%%%%%%%%%%%%

\section{Existence and Uniqueness of the solution of MB-PGF and VR-PGF}
\label{sec:ex_uni}

Let $A$ be a positive semidefinite matrix; by the spectral theorem, $A$ can be diagonalized as $A = V D V^T$, with $V$ an orthogonal matrix and $D$ a diagonal matrix with non-negative diagonal elements (the eigenvalues of $A$). We can define the principal square root $A^{1/2} := V D^{1/2} V^T$, where $D^{1/2}$ is the elementwise square root of $D$. It is clear that $A^{1/2}$ is also positive semidefinite and $A = A^{1/2} A^{1/2}.$ 

In this paper we analyze MB-PGF and VR-PGF, which we report below (see discussion and derivation in Sec.~\ref{sec:models}).
\begin{equation*}
\tag{MB-PGF}
dX(t) = -\psi(t)\nabla f(X(t)) \ dt + \psi(t)\sqrt{h/b(t)} \ \sigma_{\text{MB}}(X(t)) \ dB(t) \ \ \ \ \ \ \ \ \ \ \ \ \ \ \ \ \ \ \ \ 
\end{equation*}
\begin{equation*}
\tag{VR-PGF}
dX(t) = -\psi(t)\nabla f(X(t)) \ dt + \psi(t)\sqrt{h/b(t)} \ \sigma_{\text{VR}}(X(t),X(t-\xi(t))) \ dB(t)
\end{equation*}

The volatility of MB-SDE is defined as
$$\sigma_{\text{MB}}(x) = \left(\frac{1}{N}\sum_{i=1}^N \left(\nabla f(x) -\nabla f_{i}(x)\right)\left(\nabla f(x) -\nabla f_{i}(x)\right)^T\right)^{1/2},$$
and a similar formula holds for the $\sigma_{\text{VR}}(\cdot)$. From Thm.~\ref{thm:SDE_existence_uniqueness_global} and Thm.~\ref{thm:SDE_existence_uniqueness_SDDE}, we know that existence and uniqueness of the solution to the equations above requires this matrix valued function of $x$ to be Lipschitz continuous. Previous literature~\cite{raginsky2017non,raginsky2012continuous,mertikopoulos2018convergence,krichene2017acceleration}, put this condition as a requirement at the beginning of their analysis. However, \textit{since in our case we want to draw a direct connection to the algorithm}, we shall prove that Lipschitzianity is indeed verified.

To start, we remind again to the reader that in this paper we indicate as $\mathcal{C}^r_b(\R^d,\R^m)$ the family of $r$ times continuously differentiable functions from $\R^d$ to $\R^m$, with bounded  $r$-th derivative. If $b$ is omitted, it means we just require $f$ to be $r$ times continuously differentiable.

A crucial lemma which can be found as Prop.~6.2 in~\cite{ikeda2014stochastic} or as Thm.~5.2.3 in~\cite{stroock2007multidimensional}.

\begin{lemma}
Let $\Sigma:\R^n\to\R^{n\times n}$ be a $n\times n$ real positive semidefinite matrix function of an input vector $x\in\R^n$.
Assume each component $\Sigma_{ij}: \R^n \to \R$ be in $\mathcal{C}^2_b(\R^n,\R)$. Then, $\Sigma(x)^{1/2}$ is globally Lipschitz w.r.t. the Frobenius norm, meaning that there exists a constant $K$ such that for every $q,p \in \R^n$
    $$\left\|\Sigma(q)^{1/2}-\Sigma(p)^{1/2}\right\|\le K\|q-p\|.$$
    \label{lemma:STOC_sigma}
\end{lemma}

We proceed with the proofs of existence and uniqueness, which require the following assumption:

\textbf{(H)} Each $f_i$ is in  $\mathcal{C}^3_b(\R^d,\R)$ and is $L$-smooth.

\begin{frm-thm-app}[Existence and Uniqueness for MB-PGF] Assume \textbf{(H)}.
For all initial conditions $X(0) = x_0\in\R^d$, MB-PGF has a unique solution (in the sense of Defs.~\ref{def:SDE_sol} in App.~\ref{sec:stoc_cal}) on $[0,T]$, for any $T< \infty$. Let the stochastic process $X = \{X(t)\}_{0\le t\le T}$ be such solution; almost all (i.e. with probability $1$) realizations of $X$ are continuous functions and
$\E\left[\int_0^T \|X(t)\|^2 dt\right]< \infty.$
\end{frm-thm-app}

\begin{proof}
We basically need to check the conditions of Thm.~\ref{thm:SDE_existence_uniqueness_global}.
First, we notice that $\nabla f$ and $\sigma_{\text{MB}}$ are both Borel measurable because they are continuous.

\underline{Drift :}\ We now verify the Lipschitz condition for the drift term. For every $t\le0$ we trivially have that, since $\psi(t)\le 1$ and $f$ is $L$-smooth,
$$\|\psi(t)\nabla f(x)-\psi(t)\nabla f(y)\|\le\|\nabla f(x)-\nabla f(y)\|\le L\|x-y\|.$$
Next, we verify the linear growth condition. For every $t\ge0$, using the reverse triangle inequality and the fact that $\psi(t)\in (0,1]$ and $\psi(0) = 1$,
\begin{equation*}
        L\|x\|\ge\|\psi(t)\nabla f(x)-\psi(t)\nabla f(0_d)\|
        \ge \left(\left| \|\psi(t)\nabla f(x)\|- \|\nabla f(0_d)\|\right|\right).
\end{equation*}
Thus, we have linear growth with constant $K :=\max\left\{\|\nabla f(0)\|,L\right\}$: for every $t\in [0,T]$ and $x\in \R^d$,
$$\|\psi(t)\nabla f(x)\|\le K(1+\|x\|).$$

\underline{Volatility :} \ We need to verify the same conditions for the volatility matrix $\sigma_{\text{MB}}$. Let us define $g_i(x) := \nabla f(x) -\nabla f_{i}(x)$.
Using the definition of Frobenius norm, the linearity of $\E$, the cyclicity of the trace functional, and the fact that $\psi(t)\in (0,1]$ for all $t\ge 0$, we get

\begin{multline*}
    \|\psi(t)\sqrt{h/b(t)} \ \sigma_{\text{MB}}(x)\|^2 = \psi(t)^2\frac{h}{b(t)} \tr\left(\E \left[g_i(x)g_i(x)^T\right]\right)
    \\=\psi(t)^2\frac{h}{b(t)} \E \left[\tr (g_i(x)^T g_i(x))\right] 
    =\psi(t)^2\frac{h}{b(t)} \E \|g_i(x)\|^2.
\end{multline*}

Since $g_i(x)$ is $L$-Lipschitz, by the same argument used above for the drift term, we have $\|g_i(x)\|^2\le C(1+\|x\|^2)$ for some $C>0$ and all $i\in [N]$. Plugging this in, since $b(t)\ge 1$
\begin{equation*}
    \begin{split}
    \|\psi(t)\sqrt{h/b(t)} \ \sigma_{\text{MB}}(x)\|^2 = 
    \le D(1+\|x\|^2),
    \end{split}
\end{equation*}
for some finite positive $D$. To conclude the proof of linear growth, we notice that for any $p\in \R$, $\sqrt{1+p^2}\le 2(1+|p|)$. Thus for $B := 2D$, we have
$$\|\psi(t)\sqrt{h/b(t)} \ \sigma_{\text{MB}}(x)\|\le B(1+\|x\|).$$

Last, the global Lipschitzianity of $\sigma_{\text{MB}}$ follows directly from Lemma~\ref{lemma:STOC_sigma} using the fact that $f$ is $\mathcal{C}^3_b(\R^d,\R)$ and each $f_i$ is $\mathcal{C}^3_b(\R^d,\R)$, because then the gradients are of class $\mathcal{C}^2$ and $\sigma_{\text{MB}}$ is a smooth function of these gradients.
\
\end{proof}

\begin{frm-thm-app}[Existence and Uniqueness for VR-PGF]
Assume \textbf{(H)}. For any initial condition $x_0$, such that $X(s) = x_0$ for all $t\in[-\tau,0]$, VR-PGF has a unique solution (in the sense of Def.~\ref{def:SDDE_sol} in App.~\ref{sec:stoc_cal}) on $[-\tau,T]$, for any $T< \infty$. Moreover, let $X = \{X(t)\}_{0\le t\le T}$ be such solution; almost all realizations of $X$ are continuous functions and
$\E\left[\int_0^T \|X(t)\|^2 dt\right]< \infty.$
\end{frm-thm-app}
\begin{proof}
This time we need to check the conditions of Thm.~\ref{thm:SDE_existence_uniqueness_SDDE}. The requirements on the drift term are satisfied, as already shown in the proof for MB-PGF, since there is no delay in the drift. 
To verify the conditions on $\sigma_{\text{VR}}:\R^d\times \R^d\to \R^{d\times d}$ we proceed again as in the proof for MB-PGF, using Lemma~\ref{lemma:STOC_sigma} but this time on the joint vector $(x,\tilde x)\in \R^d\times \R^d$ ($n$ in the lemma is  $2d$), using the norm subadditivity.
\end{proof}

%%%%%%%%%%%%%%%%%%%%%%%%%%%%%%%%%%%%%%%%%%%%%%%%%%%%%%%%%%%%%%%%%%%%%%%%%%%%%%%%
%%%%%%%%%%%%%%%%%%%%%%%%%%%%%%%%%%%%%%%%%%%%%%%%%%%%%%%%%%%%%%%%%%%%%%%%%%%%%%%%

\section{Convergence proofs in continuous-time}
\label{sec:proofs_ct}

Fon convenience of the reader, we report here again the equations we are about to analyze
continuous-time models, which we analyse in this paper, are
\begin{tcolorbox}
\begin{equation*}
\tag{MB-PGF}
dX(t) = -\psi(t)\nabla f(X(t)) \ dt + \psi(t)\sqrt{h/b(t)} \ \sigma_{\text{MB}}(X(t)) \ dB(t) \ \ \ \ \ \ \ \ \ \ \ \ \ \ \ \ \ \ \ \ 
\end{equation*}
\begin{equation*}
\tag{VR-PGF}
dX(t) = -\psi(t)\nabla f(X(t)) \ dt + \psi(t)\sqrt{h/b(t)} \ \sigma_{\text{VR}}(X(t),X(t-\xi(t))) \ dB(t)
\end{equation*}
\vspace{-4mm}
\end{tcolorbox}
where 
\begin{itemize}[leftmargin=*]
\item  $\xi:\R_+\to[0,\T]$, the \textit{staleness function}, is s.t. $\xi(h k)=\xi_k$ for all $k\ge 0$;
\item  $\psi(\cdot)\in\mathcal{C}^1(\R_+,[0,1])$, the \textit{adjustment function}, is s.t. $\psi(h k)=\psi_k$ for all $k\ge 0$ and $\frac{d\psi(t)}{dt} \le 0$;
\item $b(\cdot) \in\mathcal{C}^1(\R_+,\R_+)$, the \textit{mini-batch size function} is s.t. $b(hk)=b_k$ for all $k\ge0$ and $b(t)\ge1$;
\item $\{B(t)\}_{t\ge 0}$ is a $d-$dimensional Brownian Motion on some filtered probability space.
\end{itemize}

For existence and uniqueness we need to assume the following:

\textbf{(H)} \ \ \ \ \ \ \ \ Each $f_i(\cdot)$ is in $\mathcal{C}^3$ with bounded third derivative and $L$-smooth.

We also recall some of assumptions introduced in the main paper.

\textbf{(H\begin{tiny}{WQC}\end{tiny})} \ \ \  $f(\cdot)$ is $\mathcal{C}^1$ and exists $\tau>0$ and $x^\star$ s.t. $\langle\nabla f(x),x-x^{\star}\rangle \ge \tau(f(x)-f(x^{\star}))$ for all $x\in\R^d$.

\textbf{(H\begin{tiny}{P\L{}}\end{tiny})} \ \ \ \ \ \ $f(\cdot)$ is $\mathcal{C}^1$ and there exists $\mu>0$ s.t. $\|\nabla f(x)\|^2\ge2\mu(f(x)-f(x^\star))$ for all $x\in\R^d$.

\textbf{(H\begin{tiny}{RSI}\end{tiny})} \ \ \ \ \ $f(\cdot)$ is $\mathcal{C}^1$ and there exists $\mu>0$ s.t. $\langle\nabla f(x),x-x^\star\rangle\ge\frac{\mu}{2}\|x-x^\star\|^2$ for all $x\in\R^d$.

\subsection{Supporting lemmas}
The following bound on the spectral norm also be found in~\cite{mertikopoulos2018convergence,krichene2017acceleration}. We report the proof for completeness.
\begin{lemma}
Consider two symmetric $d-$dimensional square matrices $P$ and $Q$. We have
$$\tr(PQ) \le  d\cdot \|P\|_S\cdot\|Q\|_S. $$
\label{lemma:CONV_trace_product}
\end{lemma}
\begin{proof}
Let $P_j$ and $Q_j$ be the j-th row(column) of $P$ and $Q$, respectively.
\begin{align}
    \tr(PQ) &= \sum_{j =1}^d P_j^TQ_j \le \sum_{j =1}^d \|P_j\|\cdot\|Q_j\| \nonumber \le \sum_{j =1}^d \|P\|_S\cdot\|Q\|_S = d\cdot \|P\|_S\cdot\|Q\|_S,
\end{align}
where we first used the  Cauchy-Schwarz inequality, and then the following inequality:
$$\|A\|_S = \sup_{\|z\|\le 1}\|Az\|\ge \|Ae_j\| = \|A_j\|,$$
where $e_j$ is the j-th vector of the canonical basis of $\mathbb{R}^d$.
\end{proof}

We use the previous lemma to derive another key result below.

\begin{lemma}
Assume (\textbf{H}). For any volatility matrix $\sigma(\cdot)$ such that $\|\sigma\sigma^T\|_S$ is upper bounded by $\sigma^2_*$, we have
$$\tr\left(\sigma \sigma^T\right) \le d \sigma^2_*, \quad \quad  \tr\left(\sigma\sigma^T\nabla^2 f(x)\right) \le Ld \sigma^2_*.$$
\label{lemma:CONV_trace_star_MB}
\end{lemma}
\begin{proof} We will just prove the first inequality, since the proof for the second is very similar.
\begin{equation*}
        \tr\left(\sigma \sigma^T\nabla^2 f(x)\right)\le d \|\nabla^2 f(x)\|_S \|\sigma \sigma^T\|_S
        \le Ld \sigma^2_*,
\end{equation*}
where in the equality we used the cyclicity of the trace, in the first inequality we used Lemma~\ref{lemma:CONV_trace_product} and in the last inequality we used and smoothness. 
\end{proof}

\subsection{Analysis of MB-PGF}
\label{sec:proofs_MB_PGF}
We provide a non-asymptotic analysis and then derive asymptotic rates.
\subsubsection{Non-asymptotic rates}
\label{sec:non_asy_GF}

These rates for MB-PGF are sketched in Sec.~\ref{sec:rates}. We define $\varphi(t) := \int_0^t \psi(s) ds$. As~\cite{mertikopoulos2018convergence, krichene2017acceleration}, we introduce a bound on the volatility in order to use Lemma~\ref{lemma:CONV_trace_star_MB}.

\textbf{(H$\boldsymbol{\sigma}$)} $\sigma^2_*:=\sup_{x\in\R^d}\|\sigma_{\text{MB}}(x)\sigma_{\text{MB}}(x)^T\|_S<\infty$, where $\| \cdot \|_S$ denotes the spectral norm.

\begin{frm-thm-app}[Restated Thm.~\ref{prop:CONV_GF_smooth}] Assume \textbf{(H)}, \textbf{(H$\boldsymbol{\sigma}$)}. Let $t>0$ and $\tilde t\in [0,t]$ be a random time point with distribution $\frac{\psi(\tilde t)}{\varphi(t)}$ for $\tilde t\in[0,t]$ (and $0$ otherwise). The solution to \underline{MB-PGF} is s.t.
$$\E\left[\|\nabla f(X(\tilde t))\|^2\right] \le \frac{f(x_0)-f(x^{\star})}{\varphi(t)}+ \frac{L \ d \ \sigma_*^2 h}{2 \ \varphi(t)}\int_0^t\frac{\psi(s)^2}{b(s)}ds.$$
\vspace{-3mm}
\end{frm-thm-app}

\begin{proof}
Define the energy $\EE\in\mathcal{C}^2(\R^d, \R_{+})$ such that $\EE(x) := f(x)-f(x^\star)$. First, we find a bound on the infinitesimal diffusion generator of the stochastic process $\{\EE(X(t))\}_{t\ge0}$, which generalizes the concept of derivative for stochastic systems and is formally defined in App.~\ref{subsec:ito}.
\begin{equation*}
    \begin{split}
        \mathscr{A}\EE(X(t)) & = \frac{h \psi(t)^2}{2b(t)} \tr\left( \sigma_{\text{MB}}(X(t)) \sigma_{\text{MB}}(X(t))^T \partial_{xx}\EE(X(t)) \right) + \langle\partial_{x}\EE(X(t)),-\psi(t)\nabla f(X(t))\rangle\\
        & \le \frac{h \ L \ d \ \psi(t)^2}{2b(t)}\sigma_*^2 - \psi(t)\|\nabla f(X(t))\|^2,
    \end{split}
\end{equation*}
where in the inequality we used Lemma~\ref{lemma:CONV_trace_star_MB}.

Note that the definition of $\mathscr{A}\EE(X(t))$ in Eq.~\eqref{eqn:SDE_diffusion_op} does not include the term $\langle\partial_{x}\EE,\sigma(t) dB(t) \rangle$ that vanishes when taking the expectation of the stochastic integral in Eq.~\eqref{eq:dynkin}. Therefore, integrating the bound above yields
\begin{equation}
\E[\EE(X(t),t)]-\EE(x_0,0)\le \frac{h L d \sigma_*^2}{2}\int_0^t \frac{\psi(s)^2}{b(s)} ds-\E\left[\int_0^t \psi(s)\|\nabla f(X(s))\|^2 ds\right].
\label{eq:before_trick_inf}
\end{equation} 
Next, notice that, since $\int_0^t\frac{\psi(s)}{\varphi(t)}dt = 1$, the function $s\mapsto \frac{\psi(s)}{\varphi(t)}$ defines a probability distribution. Let $\tilde t\in[0,t]$ have such distribution;  using the law of the unconscious statistician
$$\E[\|\nabla f(X(\tilde t)\|^2] = \frac{1}{\varphi(t)}\int_0^t \psi(s)\|\nabla f(X(s))\|^2 ds.$$
This trick was also used in the original SVRG paper~\cite{johnson2013accelerating}. To conclude, we plug in the last formula into Eq.~\eqref{eq:before_trick_inf}:
$$\E[\EE(X(t),t)]-\EE(x_0,0)\le \frac{h L d \sigma_*^2}{2}\int_0^t \frac{\psi(s)^2}{b(s)} ds-\varphi(t)\E\left[\|\nabla f(X(\tilde t)\|^2\right].$$
The result follows after dividing both sides by $\varphi(t)$, which is always positive for $t>0$.
\end{proof}

\begin{frm-thm-app}[Restated Thm.~\ref{prop:CONV_GF_WQC}]
Assume \textbf{(H)}, \textbf{(H$\boldsymbol{\sigma}$)}, \textbf{(H\begin{tiny}{WQC}\end{tiny})}. Let $\tilde t$ be as in Thm.~\ref{prop:CONV_GF_smooth}. The solution to \underline{MB-PGF} is s.t.
\begin{equation}
\tag{W1}
    \E \left[f(X(\tilde t)) -f(x^{\star})\right]\le \frac{\|x_0-x^{\star}\|^2}{2 \ \tau \ \varphi(t) }+ \frac{h \ d \ \sigma^2_*}{2 \ \tau \ \varphi(t)}\int_0^t\frac{\psi(s)^2}{b(s)}ds \ \ \ \ \ \ \ \ \ \ \ \ \ \ \ \ \ \ \ \ \  
\end{equation}
\begin{equation}
    \tag{W2}
    \E\left[(f(X(t))-f(x^{\star}))\right]  \le\frac{\|x_0-x^{\star}\|^2}{2 \ \tau \ \varphi(t) }+ \frac{h \ d \ \sigma^2_*}{2 \ \tau \ \varphi(t) }\int_0^t(L \ \tau \ \varphi(s)+1)\frac{\psi(s)^2}{b(s)}ds.
\end{equation}
    \vspace{-2mm}
\end{frm-thm-app}

\begin{proof} We prove the two formulas separately.

\underline{First formula.} Define the energy $\EE\in\mathcal{C}^2(\R^d, \R_{+})$ such that  $\EE(x) := \frac{1}{2}\|x-x^\star\|^2$. First, we find a bound on the infinitesimal diffusion generator of the stochastic process $\{\EE(X(t))\}_{t\ge0}$.

\begin{equation*}
    \begin{split}
        &\mathscr{A}\EE(X(t))=\\ & = \frac{h \ \psi(t)^2}{2b(t)}\tr\left( \sigma_{\text{MB}}(X(t)) \sigma_{\text{MB}}(X(t))^T \partial_{xx}\EE(X(t))\right) + \langle\partial_{x}\EE(X(t)),-\psi(t)\nabla f(X(t))\rangle\\
        & \le \frac{h \ d \ \psi(t)^2}{2b(t)}\sigma^2_* - \psi(t) \langle\nabla f(X(t)),X(t)-x^\star\rangle\\
        & \le \frac{h \ d \ \psi(t)^2}{2b(t)}\sigma^2_* - \tau \psi(t) (f(X(t))-f(x^\star)),\\        
    \end{split}
\end{equation*}

where in the first inequality we used Lemma~\ref{lemma:CONV_trace_star_MB} and in the second inequality we used weak-quasi-convexity.
Integrating this bound (see Eq.~\eqref{eq:dynkin}), we get
\begin{equation*}
    \E[\EE(X(t),t)]-\EE(x_0,0)\le\frac{h d \sigma_*^2}{2}\int_0^t \frac{\psi(s)^2}{b(s)} ds -\tau \E\left[\int_0^t\psi(s) (f(X(s))-f(x^\star)) ds\right].
\end{equation*}

Proceeding again as in the proof of Thm.~\ref{prop:CONV_GF_smooth} (above), we get the desired result.

\underline{Second formula : } Define the energy $\EE\in\mathcal{C}^2(\R^d\times \R, \R_{+})$ such that $\EE(x,t) := \tau \varphi(t) (f(x)-f(x^\star)) + \frac{1}{2}\|x-x^\star\|^2$. First, we find a bound on the infinitesimal diffusion generator of the stochastic process $\{\EE(X(t),t)\}_{t\ge0}$.
\begin{align*}
    &\mathscr{A}\EE(X(t),t)=\\
    & =\partial_t \EE(X(t),t) + \frac{h \ \psi(t)^2}{2b(t)}\tr\left( \sigma_{\text{MB}}(X(t)) \sigma_{\text{MB}}(X(t))^T \partial_{xx}\EE(X(t),t) \right)\\
    & \ \ \ + \langle\partial_{x}\EE(X(t),t),-\psi(t)\nabla f(X(t))\rangle\\
    & \le  \tau \psi(t) (f(X(t))-f(x^\star)) + \frac{h \ d \ \psi(t)^2 }{2b(t)}(L\tau\varphi(t)+1) \sigma_*^2 \\
    & \ \ \ \ +  \langle\tau \varphi(t) \nabla f(X(t)) + X(t)-x^\star, -\psi(t) \nabla f(X(t))\rangle\\
    & \le  \tau \psi(t) (f(X(t))-f(x^\star)) -\psi(t) \langle \nabla f(X(t)), X(t)-x^\star\rangle + \frac{h \ d \ \psi(t)^2 }{2b(t)}(L\tau\varphi(t)+1) \sigma_*^2\\
    &\le \frac{h \ d \ \sigma_*^2}{2}\frac{(L\tau\varphi(t)+1)\psi(t)^2}{b(t)}  ,
\end{align*}
 where in the first inequality we used the fact that $\dot\varphi(t) = \psi(t)$ and Lemma~\ref{lemma:CONV_trace_star_MB}; in the second inequality we discarded a negative term; in the third inequality we used weak-quasi-convexity.
Next, after integration (see Dynkin formula Eq.~\eqref{eq:dynkin}), plugging in the definition of $\EE$, we get
\begin{equation*}
     \tau \varphi(t) \E\left[f(X(t))-f(x^\star)\right] + \frac{1}{2}\E\left[\|X(t)-x^\star\|^2\right] \le\frac{1}{2}\|x_0-x^\star\|^2+ \frac{d h \sigma_*^2}{2}\int_0^t\frac{(L\tau\varphi(s)+1)\psi(s)^2}{b(s)}.
 \end{equation*}
Discarding the positive term $\frac{1}{2}\E\left[\|X(t)-x^\star\|^2\right]$ on the LHS and dividing\footnote{$\varphi(t)$ is the integral of $\psi(t)$, which starts positive, so it is positive for $t>0$.} everything by $\tau \varphi(t)$ we get the result. \
\end{proof}

\begin{frm-thm}[Restated Thm.~\ref{prop:CONV_GF_PL}]
Assume \textbf{(H)}, \textbf{(H$\boldsymbol{\sigma}$)}, \textbf{(H\begin{tiny}{P\L{}}\end{tiny})}. The solution to \underline{MB-PGF} is s.t.
$$\E[f(X(t))-f(x^{\star})] \le e^{-2\mu \varphi(t)}(f(x_0)-f(x^{\star})) + \frac{h \ L \ d \ \sigma^2_* }{2} \int_0^t\frac{\psi(s)^2}{b(s)} e^{-2\mu (\varphi(t)-\varphi(s))}ds.$$
\vspace{-3mm}
\end{frm-thm}

\begin{proof}
Define the energy $\EE\in\mathcal{C}^2(\R^d\times \R, \R_{+})$ such that $\EE(x,t) := e^{2\mu \varphi(t)}(f(x)-f(x^\star))$. First, we find a bound on the infinitesimal diffusion generator of the stochastic process $\{\EE(X(t),t)\}_{t\ge0}$.
\begin{equation*}
\begin{split}
     &\mathscr{A}\EE(X(t),t)=\\
     &=\partial_t \EE(X(t),t) + \frac{h \ \psi(t)^2}{2b(t)}\tr\left( \sigma_{\text{MB}}(X(t)) \sigma_{\text{MB}}(X(t))^T \partial_{xx}\EE(X(t),t) \right)\\& \ \ \ + \langle\partial_{x}\EE(X(t),t),-\psi(t)\nabla f(X(t))\rangle\\
    & \le 2\mu \ \psi(t) \ e^{2\mu \varphi(t)}(f(X(t))-f(x^\star))+
        \frac{h \ d \ L \ \psi(t)^2}{2b(t)} \sigma_*^2 e^{2\mu \varphi(t)} - \psi(t) \ e^{2\mu \varphi(t)} \|\nabla f(X(t))\|^2\\ 
        &\le \frac{h \ d \ L \ \psi(t)^2}{2b(t)} \sigma_*^2 e^{2\mu \varphi(t)},
        \\   
\end{split}
\end{equation*}
where in the first inequality we used the fact that $\dot\varphi(t) = \psi(t)$ and Lemma~\ref{lemma:CONV_trace_star_MB} and in the second inequality we used the P\L{} assumption.

Finally, after integration (see Eq.~\eqref{eq:dynkin}), plugging in the definition of $\EE$, we get
$$e^{2\mu \varphi(t)}\E[f(X(t))-f(x^\star)] \le f(x_0)-f(x^\star) + \frac{h \ d \ L \ \sigma_*^2}{2}  \int_0^t\frac{\psi(s)^2}{2b(s)} e^{2\mu \varphi(s)}ds.$$
The statement follows once we divide everything by $e^{2\mu \varphi(t)}$.
\end{proof}

\newpage

\subsection{Asymptotic rates for decreasing adjustment function}
\label{sec:proofs_MB_PGF_decrease}

\begin{corollary}Assume \textbf{(H)}, \textbf{(H$\boldsymbol{\sigma}$)}. Let $t>0$ and $\tilde t\in [0,t]$ be a random time point with distribution $\frac{\psi(\tilde t)}{\varphi(t)}$ for $\tilde t\in[0,t]$ (and $0$ otherwise). If $\psi(\cdot)$ has the form $\psi(t)=1/(t+1)^a$ and $b(t)=b\ge1$ then \underline{MB-PGF} is s.t.
        \begin{equation*}
            \E \left[\|\nabla f(X(\tilde t))\|^2\right] \le \begin{cases}
            \mathcal{O}\left(\frac{1}{t^{a}}\right) & 0<a< \frac{1}{2}\\
            \mathcal{O}\left(\frac{\log(t)}{\sqrt{t}}\right) & a = \frac{1}{2}\\
            \mathcal{O}\left(\frac{1}{t^{1-a}}\right) & \frac{1}{2}<a<1\\
            \mathcal{O}\left(\frac{1}{\log(t)}\right) & a=1
            \end{cases}.
        \end{equation*}  
\label{cor:CONV_GF_smooth_cor}
\end{corollary}
\begin{proof}
Thanks to Prop.~\ref{prop:CONV_GF_smooth}, we have
$$\E\left[\|\nabla f(X(\tilde t))\|^2\right] \le \frac{f(x_0)-f(x^{\star})}{\varphi(t)}+ \frac{L \ d \ \sigma_*^2 h}{2 \ b \ \varphi(t)}\int_0^t \psi(s)^2ds.$$
First, notice that if $a>1$, $\lim_{t\to\infty}\varphi(t)<\infty$ and we cannot retrieve convergence. Else, for $0<a<1$, the deterministic term $\frac{f(x_0)-f(x^\star)}{\varphi(t)}$ is  $\mathcal{O}\left(t^{1-a}\right)$ and $\mathcal{O}\left(\log^{-1}(t)\right)$ for $a=1$. The stochastic term $\frac{1}{\varphi(t)}\int_0^t\psi(s)^2ds$ is $\mathcal{O}\left(t^{-a}\right)$ for $a\in (0,1/2)\cup(1/2,1)$, $\mathcal{O}\left(\frac{\log(t)}{\sqrt{t}}\right)$ for $a=\frac{1}{2}$ and $\mathcal{O}(1)$ for $a=1$. The assertion follows combining asymptotic rates just derived for the deterministic and the stochastic term.
\end{proof}

\begin{corollary}
Assume \textbf{(H)}, \textbf{(H$\boldsymbol{\sigma}$)}, \textbf{(H\begin{tiny}{WQC}\end{tiny})}. Let $\tilde t$ be as in Thm.~\ref{prop:CONV_GF_smooth}. If $\psi(\cdot)$ has the form $\psi(t)=1/(t+1)^a$ and $b(t)= b \ge1$, then the solution to \underline{MB-PGF} is s.t. 
    \begin{equation*}
        \E\left[f(X(\tilde t))-f(x^\star)\right] \le \begin{cases}
        \mathcal{O}\left(\frac{1}{t^{a}}\right) & 0<a< \frac{1}{2}\\
        \mathcal{O}\left(\frac{\log(t)}{\sqrt{t}}\right) & a = \frac{1}{2}\\
        \mathcal{O}\left(\frac{1}{t^{1-a}}\right) & \frac{1}{2}<a<1\\
        \mathcal{O}\left(\frac{1}{\log(t)}\right) & a=1\\
        \end{cases}.
    \end{equation*}
        
    Moreover, for $\frac{1}{2}\le a\le 1$ we can avoid taking a randomized time point:
    
    \begin{equation*}
        \E\left[f(X(t))-f(x^\star)\right] \le \begin{cases}
        \mathcal{O}\left(\frac{1}{t^{2a-1}}\right) & \frac{1}{2}<a< \frac{2}{3}\\
        \mathcal{O}\left(\frac{\log(t)}{t^{1/3}}\right) & a = \frac{2}{3}\\
        \mathcal{O}\left(\frac{1}{t^{1-a}}\right) & \frac{2}{3}<a<1\\
        \mathcal{O}\left(\frac{1}{\log(t)}\right) & a=1\\
        \end{cases}.
    \end{equation*}    
    
\label{cor:CONV_GF_WQC_cor}
\end{corollary}

\begin{proof}
The first part is identical to Cor.~\ref{cor:CONV_GF_smooth_cor} using this time Prop.~\ref{prop:CONV_GF_WQC}.
Regarding the second part, again from Prop.~\ref{prop:CONV_GF_WQC} we have 
\begin{equation*}
        \E \left[f(X(\tilde t)) -f(x^{\star})\right]\le \frac{\|x_0-x^{\star}\|^2}{2 \ \tau \ \varphi(t) }+ \frac{h \ d \ \sigma^2_*}{2 \ \tau \ b \ \varphi(t)}\int_0^t\psi(s)^2ds.
\end{equation*}
    
The deterministic term  $\frac{1}{2 \tau \varphi(t) }\|x_0-x^\star\|^2$ is $\mathcal{O}\left(\frac{1}{t^{1-a}}\right)$ for $0<a<1$, $\mathcal{O}\left(\frac{1}{\log(t)}\right)$ for $a=1$ and $\mathcal{O}(1)$ (i.e. does not converge to 0) for $a>1$.

The stochastic term $\frac{h \ d \ \sigma^2_*}{2 \ \tau \ \varphi(t)}\int_0^t\psi(s)^2ds$
requires a more careful analysis : first of all notice that $(L\tau\varphi(t)+1)\psi(t)^2$ is $\mathcal{O}\left(\max\left\{t^{1-3a},t^{-2a}\right\}\right)$. Hence its integral is $\mathcal{O}\left(\max\left\{t^{2-3a},t^{1-2a}\right\}\right)$ for $\frac{1}{2}< a < \frac{2}{3}$, is $\mathcal{O}(1)$ for $a>\frac{2}{3}$ and has a more complicated asymptotic behavior for $a\ne \frac{1}{2},\frac{2}{3}$. First, it is clear that, since the integral is bounded for $\frac{2}{3}<a$, the asymptotic convergence rate in this case is $\mathcal{O}\left(\frac{1}{\varphi(t)}\right) = \mathcal{O}\left(\frac{1}{t^{1-a}}\right)$ for $\frac{2}{3}<a<1$ and $\mathcal{O}\left(\frac{1}{\log(t)}\right)$ for $a=1$.
Next, we get the two pathological cases out of our way:
\begin{itemize}
    \item For $a=\frac{1}{2}$ we do not have converge, since the partial integral term $\frac{d\sigma_*^2}{2\tau \varphi(t) }\int_0^tL\tau\varphi(s)\psi(s)^2ds$ is of the same order as $\varphi(t)$.
    \item For $a=\frac{2}{3}$, $\frac{d\sigma_*^2}{2\tau \varphi(t) }\int_0^t(L\tau\varphi(s)+1)\psi(s)^2ds$ is $\mathcal{O}\left(\log(t)\right)$. Hence, the resulting asymptotic bound is $\mathcal{O}\left(\frac{\log(t)}{\varphi(t)}\right)=\mathcal{O}\left(\frac{\log(t)}{t^{1/3}}\right)$.
\end{itemize}

Last, since  for $\frac{1}{2}< a < \frac{2}{3}$ the integral term is $\mathcal{O}\left(\max\left\{t^{2-3a},t^{1-2a}\right\}\right) = \mathcal{O}(t^{2-3a})$, the convergence rate is $\mathcal{O}\left(\frac{t^{2-3a}}{\varphi(t)}\right) = \mathcal{O}(t^{1-2a}).$ This completes the proof of the assertion. 
\end{proof}

\begin{remark}
    The best achievable rate in the context of the previous corollary is corresponding to $\psi(t) = \frac{1}{\sqrt{t}}$ if we look at the infimum, but is instead corresponding to $\psi(t) = \frac{1}{t^{2/3}}$ if we just look at the final point.
\end{remark}

\begin{corollary}Assume \textbf{(H)}, \textbf{(H$\boldsymbol{\sigma}$)}, \textbf{(H\begin{tiny}{P\L{}}\end{tiny})}. If $\psi(\cdot)$ has the form $\psi(t)=1/(t+1)^a$ and $b(t)=b\ge 1$, then he solution to \underline{MB-PGF} is s.t.
\begin{equation*}
    \E\left[f(X(t))-f(x^\star)\right] \le \mathcal{O}\left(\frac{1}{t^a}\right).
\end{equation*}

\label{cor:CONV_GF_PL_cor}
\end{corollary}

\begin{proof}
We start from Prop.~\ref{prop:CONV_GF_PL}:
    \begin{equation*}
        \E[f(X(t))-f(x^{\star})] \le e^{-2\mu \varphi(t)}(f(x_0)-f(x^{\star})) + \frac{h \ L \ d \ \sigma^2_* }{2 b} \int_0^t\psi(s)^2 e^{-2\mu (\varphi(t)-\varphi(s))}ds.
    \end{equation*}

For $0< a <1$, the term $e^{-2\mu \varphi(t)}$ goes down exponentially fast. Thus, we just need to consider the second addend. Let $\hat t \in [0,t]$, then

\begin{align*}
    \int_0^t \psi(s)^2 e^{-2\mu (\varphi(t)-\varphi(s))}ds &\le\int_0^{\hat t} \psi(s)^2 e^{-2\mu (\varphi(t)-\varphi(s))}ds + \int_{\hat t}^t \psi(s)^2 e^{-2\mu (\varphi(t)-\varphi(s))}ds\\
    &\le e^{-2\mu (\varphi(t)-\varphi\left(\hat t\right))}\int_0^{\hat t} \psi(s)^2 ds + \frac{\psi(\hat t)}{2\mu} \int_{\hat t}^t 2\mu\psi(s) e^{-2\mu (\varphi(t)-\varphi(s))}ds.
\end{align*}

Pick $\hat t=t/2$, notice that, since for $\psi(t) = \frac{1}{(1+t)^a}$, $\int_0^{t/2} \psi(s)^2 ds$ grows at most polynomially in $t$. Hence, first addend in the last formula decays exponentially fast. Then again we just need to consider the second addend of the last formula; in particular notice that
\begin{align*}
    \int_{\hat t}^t 2\mu\psi(s) e^{-2\mu (\varphi(t)-\varphi(s))}ds &= e^{-2\mu\varphi(t)} \int_{\hat t}^t 2\mu\psi(s) e^{2\mu\varphi(s))}ds \\
    &= e^{-2\mu\varphi(t)} \left(e^{2\mu\varphi(t)}-e^{2\mu\varphi(t/2)}\right) \\
    &= 1-e^{-2\mu(\varphi(t)-\varphi(t/2))}. \\
\end{align*}

Hence, for $t$ big enough, the considered integral will be less than $1$. All in all, we asymptotically have $\E[f(X(t))-f(x^\star)] \le \mathcal{O}(\psi(t))$, which gives the desired result.

\end{proof}

\begin{remark}
    We retrieve in continuous time the bound in~\cite{nemirovsky1983problem}: the rate is always $\Omega\left(\frac{1}{t}\right)$. 
\end{remark}

% \begin{figure}
%   \centering
%     \includegraphics[width=0.6\textwidth]{img/mb_pgf_quadratic_1_over_t.eps}
%      \caption{Bound provided by Corollary \ref{cor:CONV_GF_WQC_cor} for  $\psi(t) = \mathcal{O}(1/t^{0.95})$. Tested on a quadratic objective. Simulation with Euler-Maruyama (stepzize $=10^{-2}$).}
% 	\label{fig:bound_beta}
% \end{figure}

\subsubsection{Limit sub-optimality under constant adjustment function}
\label{sec:proofs_MB_PGF_ball}

  \begin{figure}
  \centering
  \begin{minipage}[b]{0.47\textwidth}
    \includegraphics[width=\textwidth]{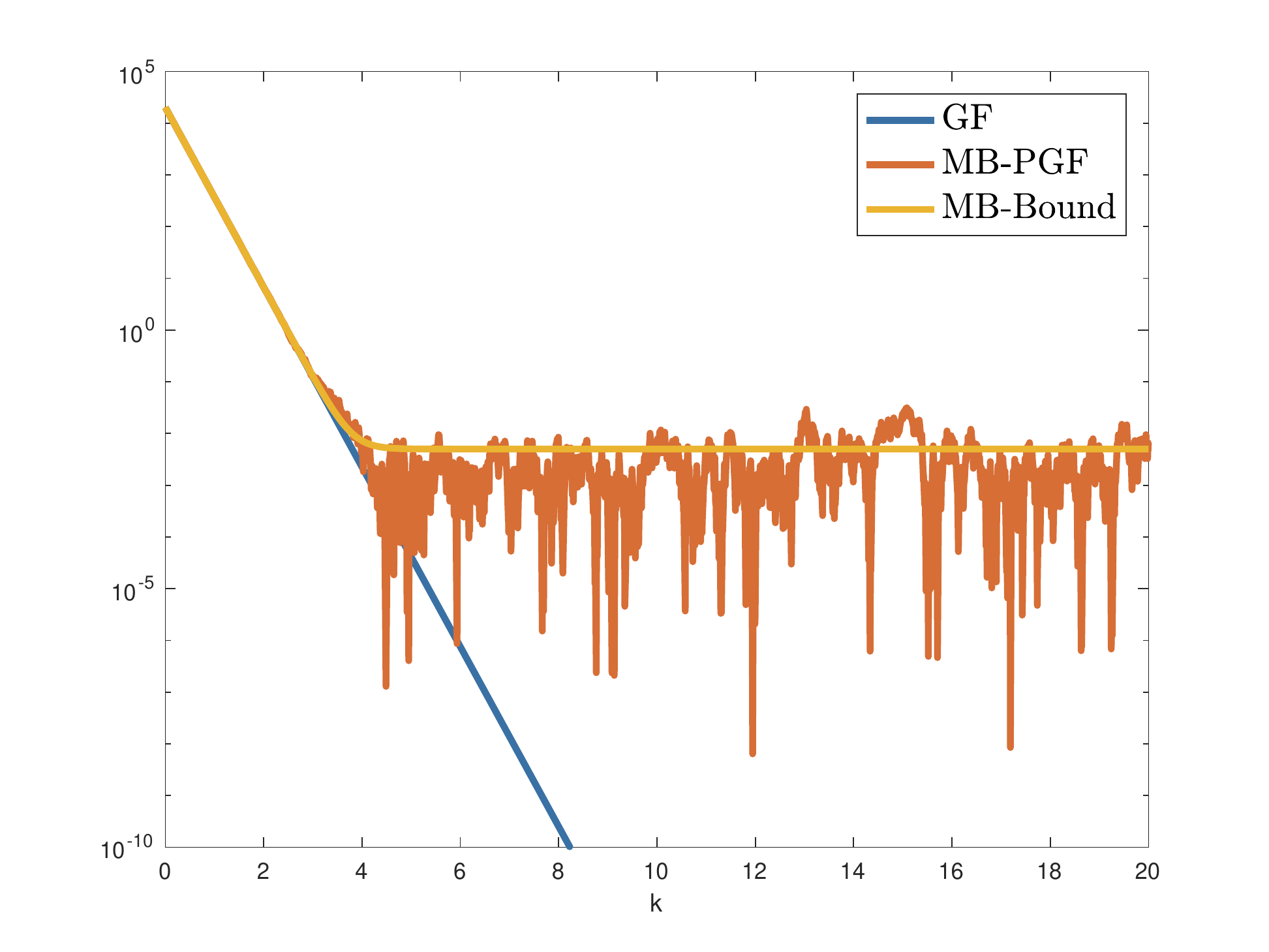}    \caption{Simulation of MB-PGF for $f(x) = \frac{1}{2} \mu\|x\|^2$ with $x\in\mathbb{R}^2$, $\sigma_*^2 = 0.1$ and $\mu = 2$. Simulation with Euler-Maruyama (stepzize $=10^{-4}$).}
    \label{fig:parabula_PL_2}
  \end{minipage}
  \hfill
  \begin{minipage}[b]{0.47\textwidth}
    \includegraphics[width=\textwidth]{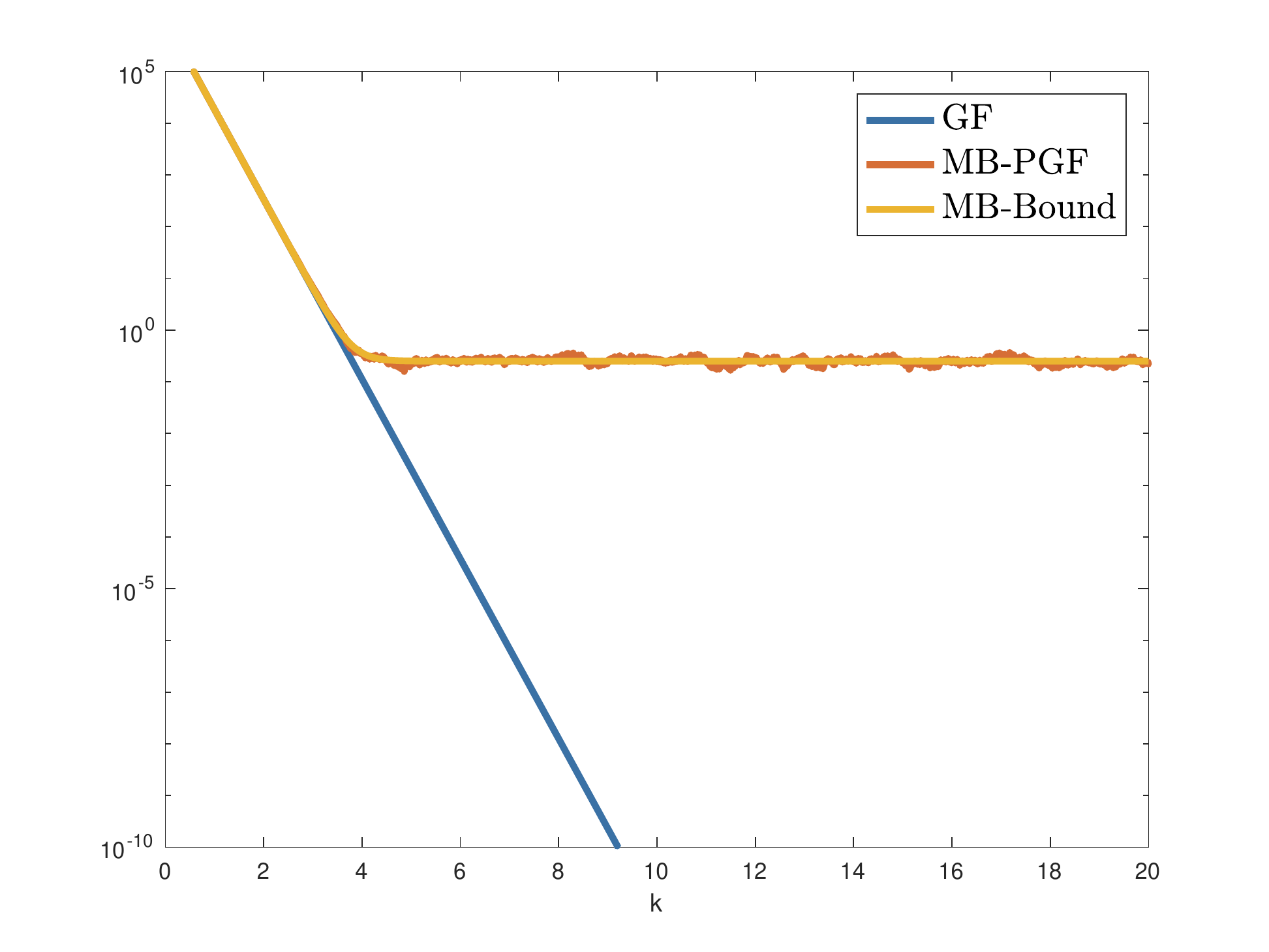}    \caption{Simulation of MB-PGF for $f(x) = \frac{1}{2} \mu \|x\|^2$ with $x\in\mathbb{R}^{100}$, $\sigma_*^2 = 0.1$ and $\mu = 2$. Simulation with Euler-Maruyama (stepzize $=10^{-4}$).}
    \label{fig:parabula_PL_100}
  \end{minipage}
\end{figure}

\begin{table*}
\centering
\begin{tabular}{| l | l | c |}
\hline
\textbf{Condition}& \textbf{Limit} & \textbf{Bound}
\\
\hline
\textbf{(H)}, \textbf{(H$\boldsymbol{\sigma}$)} & $\lim_{t\to\infty}\E \left[\|\nabla f(X(\tilde t))\|^2\right]$ &
 $\frac{L d \sigma_*^2}{2b}$ \\
  & & \\

\hline
\textbf{(H)}, \textbf{(H$\boldsymbol{\sigma}$)}, \textbf{(H\begin{tiny}{WQC}\end{tiny})}&\ $\lim_{t\to\infty}\E \left[ f(X(\tilde t)) -f(x^\star)\right]$ & $\frac{L d \sigma_*^2}{2\tau b}$ \\
& & \\
\hline
\textbf{(H)}, \textbf{(H$\boldsymbol{\sigma}$)}, \textbf{(H\begin{tiny}{P\L{}}\end{tiny})} & $\lim_{t\to\infty}\E \left[f(X(t)) -f(x^\star)\right]$  & $\frac{Ld \sigma^2_\#}{4\mu b}$ \\
& & \\
\hline 
\end{tabular}
\caption{Ball of convergence of MB-PGF under constant $\psi(t)=1$, $b(t)=b$. For $t>0$, $\tilde t\in [0,t]$ has probability distribution $\frac{\psi(s)}{\varphi(t)}$ for $s\in[0,t]$ (and $0$ otherwise).}
\label{tb:summary_rates_constant}
\end{table*}

In this paragraph we pick $\psi(t) = 1$. The results can be found in Tb.~\ref{tb:summary_rates_constant}. The only non-obvious limit is the one for P\L{} functions. By direct calculation,

\begin{equation*}
\begin{split}
\E[f(X(t))-f(x^\star)] & \le e^{-2\mu \varphi(t)}(f(x_0)-f(x^\star)) + \frac{h \ d \ L }{2b} \int_0^t\sigma_*^2 e^{-2\mu (t-s)}ds \\
&= e^{-2\mu \varphi(t)}(f(x_0)-f(x^\star)) +\frac{h \ d \ L \ \sigma_*^2}{2} \frac{1-e^{-2 \mu t}}{2\ b \ \mu}.
\end{split}
\end{equation*}
The result follows taking the limit.

% \begin{remark}
%     The ball size we just found in Corollary \ref{cor:CONV_GF_WQC_QG_constant} perfectly matches the result in \cite{bottou2018optimization} in Theorem 4.6: for $\eta_0$ small enough, SGD with noise expected norm squared bounded by $M/2$ is such that
%     $$\E[f(x_k)-f(x^\star)]\le e^{\mu k \eta_0} (f(x)-f(x^\star)) + \frac{\eta_0 L d M^2}{4\mu} $$
% \end{remark}
\begin{example}
We can verify the results in Tb.~\ref{tb:summary_rates_constant} using the quadratic function $f(x) = \frac{\mu}{2} \|x\|^2$, which is P\L{}. This function is isotropic, so $\mu = L$. Under persistent noise $\sigma_*^2 \I_d$, where $I_d$ is the identity matrix, the MB-PGF is $dX(t) = -\mu X(t) dt +h \sigma_* dB(t) $. This has solution $\mathbb{E}[f(X(t)] = f(x_0) e^{-2\mu t} + \frac{h d \sigma_*^2}{4}$, which perfectly matches the bound in Tb.~\ref{tb:summary_rates_constant}. In Fig.~\ref{fig:parabula_PL_2} and~\ref{fig:parabula_PL_100} one can see a simulation for $d = 1$ and $d = 100$, keeping the noise constant at $\sigma^2_* = 0.1$ and $\mu = 2$. One can clearly see that the bound is increasing with the number of dimensions. Moreover, by the law of large numbers, the variance in $f(X)$ is decreasing with the number of dimensions (it is a sum of $\chi^2$ distributions).
\label{ex:lower_PL}
\end{example}

\subsection{Analysis of VR-PGF}

We remind the reader that the SVRG gradient estimate (see Sec.~\ref{sec:models}), with mini-batch size $b(t)=1$ (always assumed here) is defined as
$$\mathcal{G}_{\text{VR}}(x_k) := \nabla f_{i_k}(x_k)- \nabla f_{i_k}(\tilde{x}_k)+\nabla f(\tilde{x}_k),$$
where $f(x) = \frac{1}{N}\sum_{i = 1}^N f_i(x)$, with $\{f_i\}_{i=1}^N$ a collection of functions s.t. $f_i : \mathbb{R}^d \to \mathbb{R}$ for any $i \in \{1,\cdots,N\}$. We call $x^\star$ the unique global minimum of $f$. The stochastic gradient index $i_k$ is sampled uniformly from $\{1,\dots,N\}$  and $\tilde{x}_k\in \{x_0,x_1,\dots,x_{k-1}\}$ is the pivot used at iteration $k$. SVRG builds a sequence $(x_k)_{k\ge0}$ of estimates of the solution $x^\star$ in a recursive way:
\begin{equation*}
\tag{SVRG}
x_{k+1} = x_{k} - h \mathcal{G}_{\text{VR}}(x_k,\tilde x_{k-\xi_k}),
\end{equation*}
where $h\ge 0$. $\xi_k$ is picked to be the sawtooth wave function with period $m\in\N_+$. Also, after $m$ iterations, the standard discrete-time SVRG analysis~\cite{johnson2013accelerating,reddi2016stochastic,reddi2016proximal,allen2016variance,allen2016improved} requires \textbf{"jumping"} and set $x_k = x_{\hat r_k}$, where $\hat r_k$ is picked at random from $\{k-m,\dots,k-1\}$. This is known as \textit{Option II}~\cite{johnson2013accelerating}, as opposed to \textit{Option I} which performs no jumps. The latter variant is widely used in practice~\cite{harikandeh2015stopwasting}, but, unfortunately, is not typically analyzed in the discrete-time literature.

As in App.~\ref{sec:mb_sgd_nonasy}, we denote by $\{\mathcal{F}_k\}_{k\ge 0}$ the natural filtration induced by the stochastic process with jumps $\{x_k\}_{k\ge 0}$. The conditional mean and covariance matrix of $\mathcal{G}_{\text{VR}}$ are
\begin{align}
&\E_{\F_{k-1}}\left[\mathcal{G}_{\text{VR}}(x_k)\right]  = \nabla f(x_k),  \\
&\Sigma_{\text{VR}}(x_k,\tilde{x}_k) := \var_{\F_{k-1}}\left[\mathcal{G}_{\text{VR}}(x_k)\right] \\
& \ \ \ = \E_{\mathcal{F}_{k-1}}\left[ \left(\mathcal{G}_{\text{VR}}(x_k)-\nabla f(x_k)\right)\left(\mathcal{G}_{\text{VR}}(x_k)-\nabla f(x_k)\right)^T\right]\nonumber.
\end{align}

We start with a lemma and a corollary, which will be used both in continuous and in discrete time and that are partially derived in~\cite{johnson2013accelerating} and~\cite{allen2016improved}.

\begin{lemma} Assume \textbf{(H)}. We have
\begin{multline*}
    \tr\left(\Sigma_{\text{VR}}(x_k,\tilde{x}_k)\right)\le\E_{\mathcal{F}_{k-1}}\left[\|\mathcal{G}_{\text{VR}}(x_k)\|^2\right) \\
    \le 2\E_{\mathcal{F}_{k-1}}\|\nabla f_{i}(x_k)- \nabla f_{i}(x^\star)\|^2 + 2\E_{\mathcal{F}_{k-1}}\|\nabla f_{i}(\tilde{x}_k)-\nabla f_{i}(x^\star))\|^2.
\end{multline*}
\label{lemma:VR_trace_lemma}
\end{lemma}
\begin{proof}
Let us define $\epsilon_{\text{VR}}(x_k,\tilde x_k) := \mathcal{G}_{\text{VR}}(x_k,\tilde x_k)- \nabla f(x_k)$. First notice that, $\epsilon_{\text{VR}}(x_k,\tilde x_k)$ has zero mean, and 
\begin{multline}
 \tr\left(\Sigma_{\text{VR}}(x_k,\tilde{x}_k)\right) = \tr\left(\E_{\mathcal{F}_{k-1}}[\epsilon_{\text{VR}}(x_k,\tilde x_k)\epsilon_{\text{VR}}(x_k,\tilde x_k)^T]\right) =\\
        =\E_{\mathcal{F}_{k-1}}\left[\tr (\epsilon_{\text{VR}}(x_k,\tilde x_k)\epsilon_{\text{VR}}(x_k,\tilde x_k)^T\right]\\
        =\E_{\mathcal{F}_{k-1}}\left[\tr (\epsilon_{\text{VR}}(x_k,\tilde x_k)^T\epsilon_{\text{VR}}(x_k,\tilde x_k))\right] =\E_{\mathcal{F}_{k-1}} \|\epsilon_{\text{VR}}(x_k,\tilde x_k)\|^2,
\end{multline}

where the second equality is given by the linearity of the trace and third equality by the cyclic property of the trace. 
Notice that, since for any random variable $\zeta$ we have $\E_{\mathcal{F}_{k-1}}[\|\zeta-\E_{\mathcal{F}_{k-1}} \zeta\|^2]= \E_{\mathcal{F}_{k-1}}[\|\zeta\|^2]-\|\E_{\mathcal{F}_{k-1}} [\zeta]\|^2\le \E_{\mathcal{F}_{k-1}}[\|\zeta\|^2]$, then
$$\E_{\mathcal{F}_{k-1}} [\|\epsilon_{\text{VR}}(x_k,\tilde x_k)\|^2] \le \E_{\mathcal{F}_{k-1}} [\|\mathcal{G}_{\text{VR}}(x_k,\tilde x_k)\|^2].$$

Hence, we found that $\tr\left(\Sigma_{\text{VR}}(x_k,\tilde{x}_k)\right)\le \E_{\mathcal{F}_{k-1}}[\|\mathcal{G}_{\text{VR}}(x)\|^2]$. We further bound this term with a simple calculation

\begin{equation}
    \begin{split}
        \E_{\mathcal{F}_{k-1}} \|\mathcal{G}_{\text{VR}}(x)\|^2 &= \E_{\mathcal{F}_{k-1}}\|\nabla f_{i}(x_k)- \nabla f_{i}(\tilde{x}_k)+\nabla f(\tilde{x}_k)\|^2\\
        &=\E_{\mathcal{F}_{k-1}}\|\nabla f_{i}(x_k)- \nabla f_{i}(x^\star) - [\nabla f_{i}(\tilde{x}_k)-\nabla f_{i}(x^\star) -\nabla f(\tilde{x}_k)]\|^2\\
        &\le 2\E_{\mathcal{F}_{k-1}}\|\nabla f_{i}(x_k)- \nabla f_{i}(x^\star)\|^2 + 2\E\|\nabla f_{i}(\tilde{x}_k)-\nabla f_{i}(x^\star) -\nabla f(\tilde{x}_k)\|^2\\
        &= 2\E_{\mathcal{F}_{k-1}}\|\nabla f_{i}(x_k)- \nabla f_{i}(x^\star)\|^2\\
        & \ \ \ + 2\E_{\mathcal{F}_{k-1}}\|\nabla f_{i}(\tilde{x}_k)-\nabla f_{i}(x^\star) - \E_{\mathcal{F}_{k-1}} [\nabla f(\tilde{x}_k)-\nabla f_{i}(x^\star)]\|^2\\
        &\le 2\E_{\mathcal{F}_{k-1}}\|\nabla f_{i}(x_k)- \nabla f_{i}(x^\star)\|^2 + 2\E_{\mathcal{F}_{k-1}}\|\nabla f_{i}(\tilde{x}_k)-\nabla f_{i}(x^\star))\|^2,\\  
    \end{split}
\end{equation}
where in the first inequality we used the parallelogram law; in the third equality we used $\E_{\mathcal{F}_{k-1}}[\nabla f_i(x^\star)] = 0$ and in the second inequality we used again the fact that for any random variable $\zeta$, $\E_{\mathcal{F}_{k-1}} \|\zeta-\E_{\mathcal{F}_{k-1}} \zeta\|^2= \E_{\mathcal{F}_{k-1}}\|\zeta\|^2-\|\E_{\mathcal{F}_{k-1}} \zeta\|^2\le \E_{\mathcal{F}_{k-1}}\|\zeta\|^2.$
\end{proof}

Using the previous lemma, we can derive the following result.

\begin{corollary}
Assume \textbf{(H)}. Then
\begin{multline*}
    \tr\left(\Sigma_{\text{VR}}(x_k,\tilde{x}_k)\right)\le\E_{\mathcal{F}_{k-1}}\left[\|\mathcal{G}_{\text{VR}}(x_k)\|^2\right] \\\le 2L^2 \E_{\mathcal{F}_{k-1}}\left[\|x_k-x^\star\|^2\right] + 2L^2 \E_{\mathcal{F}_{k-1}}\left[\|\tilde{x}_k-x^\star\|^2\right].
\end{multline*}
\label{cor:STOC_SVRG_var_QG}
\end{corollary}
\begin{proof} Using first smoothness we have, starting from Lemma~\ref{lemma:VR_trace_lemma}

\begin{equation*}
    \begin{split}
        &\tr\left(\Sigma_{\text{VR}}(x_k,\tilde{x}_k)\right)\le\E_{\mathcal{F}_{k-1}}\left[\|\mathcal{G}_{\text{VR}}(x_k)\|^2\right)\\ &\le 2\E_{\mathcal{F}_{k-1}}\|\nabla f_{i}(x_k)- \nabla f_{i}(x^\star)\|^2 + 2\E_{\mathcal{F}_{k-1}}\|\nabla f_{i}(\tilde{x}_k)-\nabla f_{i}(x^\star))\|^2\\  
        &\le 2L^2 \E_{\mathcal{F}_{k-1}}\left[\|x_k-x^\star\|^2\right] + 2L^2 \E_{\mathcal{F}_{k-1}}\left[\|\tilde{x}_k-x^\star\|^2\right].\\  
    \end{split}
\end{equation*}
\end{proof}

Next, we provide a convergence rate for \textit{Option II}.

\subsubsection{Convergence rate under Option II}

We consider, the case $b(t)=\psi(t)=1$. Therefore, VR-PGF reads

$$dX(t) = -\nabla f(X(t)) \ dt + \sqrt{h} \ \sigma_{\text{VR}}(X(t),X(t-\xi(t))) \ dB(t).$$

As for standard SVRG with Option II, every $\T$ seconds we perform a jump.

\begin{frm-thm}[Restated Thm.~\ref{prop:CONV_VR_RSI}]
Assume \textbf{(H)}, \textbf{(H\begin{tiny}{RSI}\end{tiny})} and choose $\xi(t) = t-\sum_{j=1}^\infty\delta(t-j\T)$~(sawtooth wave), where $\delta(\cdot)$ is the Dirac delta. Let $\{X(t)\}_{t\ge0}$ be the solution to \underline{VR-PGF} with additional jumps at times $(j\T)_{j\in\N}$: we pick $X(j\T+\T)$ uniformly in $\{X(s)\}_{j\T\le s< (j+1)\T}$. Then,
\vspace{-2mm}
$$\E[\|X(j\T)-x^\star\|^2]=\left(\frac{2 h L^2\T + 1}{\T(\mu-2hL^2)}\right)^j \|x_0-x^*\|^2.$$
\end{frm-thm}
\begin{proof}
Define the energy $\EE\in\mathcal{C}^2(\R^d, \R_{+})$ such that $\EE(x) := \frac{1}{2}||x-x^\star||^2$. First, we find a bound on the infinitesimal diffusion generator of the stochastic process $\{\EE(X(s))\}_{j\T\le s \le (j+1)\T}$:

\begin{equation*}
    \begin{split}
        \mathscr{A}\EE(X(s)) &= -\langle\nabla f(X(s)),X(s)-x^\star\rangle ds + \frac{h}{2} \tr(\Sigma_{\text{VR}}(X(s),X(s-\xi(s)))) ds\\
        &\le -\frac{\mu}{2}\|X(s)-x^\star\|^2ds + hL^2\left(\|X(s)-x^\star\|^2 + \|X(s-\xi(s))-x^\star\|^2\right)ds\\ 
    \end{split}
\end{equation*}
where in the first inequality we used Lemma~\ref{lemma:VR_trace_lemma} and the RSI. Using Dynkin's formula (Eq.~\eqref{eq:dynkin}), since $X(s-\xi(s)) = X(j\T)$ for $s\in[j\T,j\T+\T]$ by our choice of $\xi(\cdot)$,
\begin{multline*}
   \frac{1}{2}\E\left[\|X(j\T+\T)-x^\star\|^2\right]-\frac{1}{2}\E\left[\|X(j\T)-x^\star\|^2\right]\\\le-\frac{\T}{2}(\mu-2hL^2)\int_{j\T}^{j\T+\T}\E[\|X(s)-x^\star\|^2]\frac{ds}{\T} + h L^2\T\E[\|X(j\T)-x^\star\|^2],
\end{multline*}

which gives

$$\int_{j\T}^{j\T+\T}\E[\|X(s)-x^\star\|^2]\frac{ds}{\T} \le \frac{2 h L^2\T + 1}{\T(\mu-2hL^2)}\E[\|X(j\T)-x^\star\|^2].$$

By redefining (\textit{jumping to}) $X(j\T+\T)$ uniformly from $\{X(s)\}_{j\T\le s\le j\T+\T}$, $\E[\|X(j\T+\T)-x^\star\|^2] = \int_{j\T}^{j\T+\T}\E[\|X(s)-x^\star\|^2]\frac{ds}{\T}$ and therefore, for all $j\in\N$

$$\E[\|X(j\T+\T)-x^\star\|^2]\le\frac{2 h L^2\T + 1}{\T(\mu-2hL^2)}\E[\|X(j\T)-x^\star\|^2].$$

\end{proof}

%%%%%%%%%%%%%%%%%%%%%%%%%%%%%%%%%%%%%%%%%%%%%%%%%%%%%%%%%%%%%%%%%%%%%%%%%%%%%%%%
%%%%%%%%%%%%%%%%%%%%%%%%%%%%%%%%%%%%%%%%%%%%%%%%%%%%%%%%%%%%%%%%%%%%%%%%%%%%%%%%

\section{Analysis in discrete-time}
\label{sec:proofs_dt}
For ease of consultation of this appendix, we briefly describe here again our setting: $\{f_i\}_{i=1}^N$ is a collection of $L$-smooth\footnote{As already mentioned in the main paper, we say a function $f\in\mathcal{C}^1(\R^d,\R^m)$ is $L$-smooth if, for all $x,y\in \R^d$, we have $\|\nabla f(x)-\nabla f(y)\|\le L\|x-y\|$} functions s.t. $f_i : \mathbb{R}^d \to \mathbb{R}$ for any $i \in \{1,\dots,N\}$  and $
    f(\cdot):= \frac{1}{N}\sum_{i = 1}^N f_i(\cdot)$.
 Trivially, $f(\cdot)$ is also $L$-smooth; our task is to find a minimizer $x^\star = \argmin_{x\in\R^d} f(x)$. 
 
 \textbf{(H-)} \ \ \ \ \ \ \ \ Each $f_i(\cdot)$ is $L$-smooth.

 Mini-Batch SGD builds a sequence of estimates of the solution $x^\star$ in a recursive way, using the stochastic gradient estimate $\mathcal{G}_{\text{MB}}$:
\begin{equation*}
\tag{SGD}
x_{k+1} = x_{k} - \eta_k \mathcal{G}\left(\{x_i\}_{0\le i\le k},k\right),
\end{equation*}
where $(\eta_k)_{k\ge0}$ is a non-increasing deterministic sequence of positive numbers called the \textit{learning rate sequence}. We define, as in Sec.~\ref{sec:models},
\begin{itemize}
    \item $h:=\eta_0$.
    \item \textit{adjustment factor} sequence $(\psi_k)_{k\ge 0}$ s.t. for all $k\ge 0$, $\psi_k=\eta_k/ h$.
    \item $\{\mathcal{F}_k\}_{k\ge 0}$ the natural filtration induced by the stochastic process $\{x_k\}_{k\ge 0}$.
    \item $\E$ the expectation operator over all the information $\mathcal{F}_{\infty}$.
    \item $\E_{\mathcal{F}_{k}}$ the conditional expectation given the information at step $k$.
\end{itemize}

We also report from the main paper some assumptions we might use

\textbf{(H\begin{tiny}{WQC}\end{tiny})} \ \ \  $f(\cdot)$ is $\mathcal{C}^1$ and exists $\tau>0$ and $x^\star$ s.t. $\langle\nabla f(x),x-x^{\star}\rangle \ge \tau(f(x)-f(x^{\star}))$ for all $x\in\R^d$.

\textbf{(H\begin{tiny}{P\L{}}\end{tiny})} \ \ \ \ \ \ $f(\cdot)$ is $\mathcal{C}^1$ and there exists $\mu>0$ s.t. $\|\nabla f(x)\|^2\ge2\mu(f(x)-f(x^\star))$ for all $x\in\R^d$.

\textbf{(H\begin{tiny}{RSI}\end{tiny})} \ \ \ \ \ $f(\cdot)$ is $\mathcal{C}^1$ and there exists $\mu>0$ s.t. $\langle\nabla f(x),x-x^\star\rangle\ge\frac{\mu}{2}\|x-x^\star\|^2$ for all $x\in\R^d$.

\subsection{Analysis of MB-SGD}
\subsubsection{Non-asymptotic rates}
\label{sec:mb_sgd_nonasy}
In Sec.~\ref{sec:models}, we defined $\Sigma_{\text{MB}}(x)$ to be the one-sample conditional covariance matrix. So that $\var_{\F_{k-1}}[\mathcal{G}_{\text{MB}}(x_k,k)] = \frac{\Sigma_{\text{MB}}(x_k)}{b_k}$, where $b_k$ is the mini-batch size. As commonly done in the literature~\cite{ghadimi2013stochastic} and to match the continuous time analysis, we make the following assumption.

\textbf{(H$\boldsymbol{\sigma}$)} $\sigma^2_*:=\sup_{x\in\R^d}\|\sigma_{\text{MB}}(x)\sigma_{\text{MB}}(x)^T\|_S<\infty$, where $\| \cdot \|_S$ denotes the spectral norm.

Last, we define --- to match existing proofs of related results~\cite{bottou2018optimization,ghadimi2013stochastic,moulines2011non}, $\epsilon_k:=\mathcal{G}_{\text{MB}}(x_k,k)-\nabla f(x_k)$. It follows that $\E[\|\epsilon_k\|^2] = \frac{d \sigma^2_*}{b_k}.$

Moreover for $k \ge 0$ we define $\varphi_{k+1}= \sum_{i=0}^k \psi_i$. We are now ready to show the non-asymptotic results. But first, we need two (well-known) classic lemmas.

\begin{lemma}
Assume  \textbf{(H-)}, then
$$\E [f(x_{k+1})-f(x_k)] \le \left(\frac{L\eta_k^2}{2}-\eta_k\right) \E\left[\|\nabla f (x_k)\|^2 \right]+\frac{L \ d \ \sigma^2_* \ \eta_k^2 }{2 \ b_k}.$$
\label{lemma:discrete_time_1}
\end{lemma}

\begin{proof}
Thanks to the $L$-smoothness assumption, we have the classic result (see e.g.~\cite{nesterov2018lectures})
\begin{equation}
f(x_{k+1})-f(x_k) \le \langle\nabla f (x_k), x_{k+1}-x_k \rangle +\frac{L}{2} \|x_{k+1}-x_k\|^2 \quad \quad a.s.
\label{eq:smooth_upper_nesterov}
\end{equation}
After plugging the definition of mini-batch SGD, taking the expectation and using Fubini's theorem,

\begin{align*}
&\E [f(x_{k+1})-f(x_k)]\\ &\le -\eta_k \E\left[\E_{\mathcal{F}_{k-1}}[\langle\nabla f (x_k), \mathcal{G}_{\text{MB}}(x_k,k)\rangle]\right] +\frac{L\eta_k^2}{2} \E[\|\mathcal{G}_{\text{MB}}(x_k,k)\|^2]\\
&\le -\eta_k \E[\langle\nabla f (x_k), \E_{\mathcal{F}_{k-1}}[\mathcal{G}_{\text{MB}}(x_k,k)]\rangle] +\frac{L\eta_k^2}{2} \E[\|\nabla f (x_k)+\epsilon_k\|^2]\\
&\le -\eta_k \E\left[\|\nabla f (x_k)\|^2\right] +\frac{L\eta_k^2}{2} \E\left[\|\nabla f (x_k)\|^2+\|\epsilon_k\|^2 + 2\langle\epsilon_k,\nabla f (x_k) \rangle\right]\\
&\le \left(\frac{L\eta_k^2}{2} -\eta_k\right) \E\left[\|\nabla f (x_k)\|^2\right]+\frac{L\eta_k^2}{2}\E_{\mathcal{F}_{k-1}}\left[\|\epsilon_k\|^2\right] +L\eta_k^2\E\left[\langle\E_{\mathcal{F}_{k-1}}[\epsilon_k],\nabla f (x_k) \rangle\right]\\
&\le \left(\frac{L\eta_k^2}{2}-\eta_k\right)\E\left[\|\nabla f (x_k)\|^2\right] +\frac{L d \sigma^2_* \eta_k^2 }{2 b_k}.
\end{align*}
\end{proof}
\begin{lemma} Assume \textbf{(H-)}, then
$$\E\left[\|\nabla f(x_k)\|^2\right]\le 2L \E[f(x_k)-f(x^\star)].$$
\label{lemma:discrete_time_2}
\end{lemma}
\begin{proof}
We have that
$$\E[f(x^\star)-f(x_k)]\le \E\left[ f\left(x_k-\frac{1}{L} \nabla f(x_k)\right)-f(x_k)\right]\le -\frac{1}{2L}\E\left[\|\nabla f (x_k)\|^2\right],$$
where the first inequality holds since $x^\star$ is the minimum and the last inequality uses Lemma~\ref{lemma:discrete_time_1} in the special case $\sigma^2_*=0$.
\end{proof}

The following theorem (statement and proof technique) has to be compared to Thm.~\ref{prop:CONV_GF_smooth} for MB-PGF.
 
 \begin{frm-thm-app}
Assume \textbf{(H-)}, \textbf{(H$\boldsymbol{\sigma}$)}. For $k \ge 0$ let $\tilde k\in[0,k]$ be a random index picked with probability $\psi_{j}/{\varphi_{j+1}}$ for $j \in\{0,\dots,k\}$ (and $0$ otherwise). If $h\le \frac{1}{L}$, then we have:
 $$\E\left[ \|\nabla f (x_{\tilde k})\|^2\right]\le \frac{2 \ (f(x_0)-f(x^\star))}{(h  \varphi_{k+1})}+\frac{h\  d \ L\ \sigma^2_*}{(h \varphi_{k+1})}\sum_{i=0}^k \frac{\psi_i^2}{b_i} h.$$
 \label{prop:CONV_GD_smooth}
 \end{frm-thm-app}
 
 \begin{proof}
  Consider the continuous-time inspired (see Thm.~\ref{prop:CONV_GF_smooth}) Lyapunov function $\EE(k):=f(x_k)-f(x^\star)$. We have, directly from Lemma~\ref{lemma:discrete_time_1} and using the fact that $\eta_k\le\frac{1}{L}$ (hence $\frac{L\eta_k^2}{2}-\eta_k \le -\frac{\eta_k}{2}$ ),
 \begin{align*}
\E[\EE(k+1)-\EE(k)] &= \E [f(x_{k+1})-f(x_k)] \\
&\le \left(\frac{L\eta_k^2}{2}-\eta_k\right) \E\left[\|\nabla f (x_k)\|^2 \right]+\frac{L d \sigma^2_* \eta_k^2 }{2b_k}\\
&\le -\frac{\eta_k}{2} \E\left[\|\nabla f (x_k)\|^2 \right]+\frac{L d \sigma^2_* \eta_k^2 }{2b_k}
 \end{align*}
 
 Finally, by linearity of integration,

\begin{align}
\E [\EE(k+1)-\EE(0)] &= \E \left[\sum_{i=0}^k \EE(i+1)-\EE(i)\right] \nonumber\\
&=  \sum_{i=0}^k\E[ \EE(i+1)-\EE(i)] \nonumber\\
&= -\frac{1}{2} \sum_{i=0}^k \eta_i  \E\left[\|\nabla f (x_i)\|^2 \right]+\frac{L \ d \ \sigma^2_*}{2}\sum_{i=0}^k  \frac{\eta_i^2}{b_i}\nonumber \\
&= -\frac{h}{2} \E\left[\sum_{i=0}^k \psi_i  \|\nabla f (x_i)\|^2 \right]+\frac{L \ d \ h \ \sigma^2_* }{2}\sum_{i=0}^k 
\frac{\psi_i^2}{b_i}
\label{eq:before_trick_discrete}
\end{align}
Next, notice that, since $\sum_{i=0}^k\frac{\psi_i}{\varphi_{k+1}} = 1$, the function $i\mapsto \frac{\psi_i}{\varphi_{k+1}}$ defines a probability distribution. Let $\tilde k\in\{0,\dots,k\}$ have this distribution; then conditioning on all the past iterations $\{x_0,\dots, x_k\}$ and using the law of the unconscious statistician 
$$\E_{\mathcal{F}_{k-1}}[\|\nabla f(x_{\tilde k})\|^2] = \frac{1}{\varphi_{k+1}}\sum_{i=0}^k\psi(i)\|\nabla f(x_i)\|^2 ds,$$

which, once plugged in Eq.~\eqref{eq:before_trick_discrete}, gives $$h\varphi_{k+1}\E [\|\nabla f(x_{\tilde k})\|^2]\le 2\EE(0) + \frac{L \ d \ h^2 \ \sigma^2_* }{2}\sum_{i=0}^k 
\frac{\psi_i^2}{b_i}.$$
The proof ends by using the definition of $\EE$.
 \end{proof}
 The following proposition has to be compared to Thm.~\ref{prop:CONV_GF_WQC} for MB-PGF.
 
 \begin{frm-thm-app}
 Assume \textbf{(H-)}, \textbf{(H$\boldsymbol{\sigma}$)}, \textbf{(H\begin{tiny}{WQC}\end{tiny})} and let $\tilde k$ be defined as in Thm.~\ref{prop:CONV_GD_smooth}. If $0< h\le \frac{\tau}{2L}$, then we have:
 $$\E\left[f(x_{\tilde k})-f(x^\star)]\right]\le \frac{\|x_0-x^\star\|^2}{\tau \ (h \varphi_{k+1}) }+\frac{d \ h \ \sigma^2_*}{\tau \ (h \varphi_{k+1}) }\sum_{i=0}^k \frac{\psi_i^2}{b_i}h.$$
 
 Moreover, if $0\le h \le \left(\frac{2}{L}-\frac{1}{\tau L}\right)$, then for all $k\ge 0$ we have:
 
 $$\E\left[f(x_{k+1})-f(x^\star)\right]\le \frac{\|x_0-x^\star\|^2}{2 \ \tau \ (h\varphi_{k+1})} + \frac{h \ d \ \sigma^2_*}{2 \ \tau \ (h\varphi_{k+1})} \sum_{i=0}^{k} (1+\tau \varphi_{i+1}L) \frac{\psi_i^2}{b_i}h .$$
 \label{prop:CONV_GD_WQC}
 \end{frm-thm-app}
 
 \begin{proof} We prove the two rates separately.
 
 \underline{Proof of the first formula : }
 consider the continuous-time inspired (see Thm.~\ref{prop:CONV_GF_WQC}) Lyapunov function $\EE(k):=\frac{1}{2}\|x_k -x^\star\|^2$. We have
 \begin{align*}
 &\E [\EE(k+1)-\EE(k)]=\\ &= \frac{1}{2}\E \left[\|x_{k} -x^\star-\eta_k \mathcal{G}_{\text{MB}}(x_k,k)\|^2\right] -  \frac{1}{2}\E\left[\|x_{k} -x^\star\|^2\right]\\
 &= -\eta_k\E\left[\E_{\mathcal{F}_{k-1}}\left[\left\langle\mathcal{G}_{\text{MB}}(x_k,k), x_k-x^\star \right\rangle\right]\right] +\frac{\eta_k^2}{2}\E \left[\|\mathcal{G}_{\text{MB}}(x_k,k)\|^2\right] \\
  &= -\eta_k\E\left[\left\langle\E_{\mathcal{F}_{k-1}}[\mathcal{G}_{\text{MB}}(x_k,k)], x_k-x^\star \right\rangle\right] + \frac{\eta_k^2}{2} \E\left[\|\nabla f(x_k)\|^2\right]+\frac{\eta_k^2}{2}\E_{\mathcal{F}_{k-1}} \left[\|\epsilon_k\|^2\right] \\
  &= -\eta_k\E\left[\left\langle\nabla f(x_k), x_k-x^\star \right\rangle\right] + \frac{\eta_k^2}{2} \E\left[\|\nabla f(x_k)\|^2\right]+\frac{d \eta_k^2 \sigma^2_* }{2 b_k}, 
   \end{align*}
 where in the second equality we used Fubini's theorem. We proceed using weak-quasi-convexity and Lemma~\ref{lemma:discrete_time_2}:

\begin{align*}
  \E [\EE(k+1)-\EE(k)]&\le -\eta_k \tau \E\left[ f(x_k)-f(x^\star)\right] + \frac{\eta_k^2}{2} \E\left[\|\nabla f(x_k)\|^2\right]+\frac{d \eta_k^2 \sigma^2_* }{2 b_k} \\  
    &\le -\eta_k \tau \E\left[ f(x_k)-f(x^\star)\right] + \eta_k^2L \E[f(x_k)-f(x^\star)] +\frac{d \eta_k^2 \sigma^2_* }{2 b_k} \\  
        &\le (L\eta_k^2-\tau\eta_k)  \E\left[ f(x_k)-f(x^\star)\right] +\frac{d \eta_k^2 \sigma^2_*}{2 b_k}.\\  
 \end{align*}
 Next, using the fact that $-\tau \eta_k + L\eta_k^2 \le -\tau \frac{\eta_k}{2}$ for $\eta_k\le \frac{\tau}{2L}$ we get
 \begin{equation}
 \E [\EE(k+1)-\EE(k)] \le -\frac{\eta_k\tau}{2}\E[f(x_k)-f(x^\star)]+\frac{\eta_k^2\sigma^2_* d}{2 b_k}
 \end{equation}
 
 Finally, by linearity of integration,

\begin{align*}
\E [\EE(k+1)-\EE(0)] &= \E \left[\sum_{i=0}^k \EE(i+1)-\EE(i)\right] \\
&=  \sum_{i=0}^k\E[ \EE(i+1)-\EE(i)] \\
&\le -\frac{\tau}{2} \sum_{i=0}^k\eta_i\E \left[(f(x_i)-f(x^\star))\right]+\frac{d \sigma^2_*}{2}\sum_{i=0}^k\frac{\eta_i^2}{b_i} \\
&= -\frac{\tau}{2} \E\left[\sum_{i=0}^k\eta_i(f(x_i)-f(x^\star))\right]+\frac{d \sigma^2_*}{2}\sum_{i=0}^k\frac{\eta_i^2}{b_i}.\\
\end{align*}

Proceeding again as in Thm.~\ref{prop:CONV_GD_smooth}, we get the desired result.

\underline{Proof of the second formula : }
consider the continuous-time inspired (see Thm.~\ref{prop:CONV_GF_WQC}) Lyapunov function
$$\EE(k) := \tau h \varphi_{k} (f(x_k)-f(x^\star))+\frac{1}{2} \|x_k-x^\star\|^2.$$
Then, with probability one,
\begin{align*}
&\EE(k+1)-\EE(k)\\
&=\tau h\varphi_{k+1} (f(x_{k+1})-f(x^\star))+\frac{1}{2} \|x_{k+1}-x^\star\|^2-\tau h\varphi_k (f(x_k)-f(x^\star))-\frac{1}{2} \|x_k-x^\star\|^2\\
&=\tau h\varphi_{k+1} (f(x_{k+1})-f(x_k)) + \tau \eta_{k} (f(x_k)-f(x^\star))\\
& \ \ \ \ \ + \frac{1}{2} \|x_{k} - x^\star-\eta_k\mathcal{G}_{\text{MB}}(x_k,k)\|^2-\frac{1}{2} \|x_k-x^\star\|^2\\
&=\tau \varphi_{k+1} (f(x_{k+1})-f(x_k)) + \tau \eta_{k} (f(x_k)-f(x^\star))\\
& \ \ \ \ \ + \frac{\eta_k^2}{2} \|\mathcal{G}_{\text{MB}}(x_k,k)\|^2  - \eta_k\langle \mathcal{G}_{\text{MB}}(x_k,k),x_k-x^\star \rangle.\\
&\le  \tau \eta_0 \varphi_{k+1} (f(x_{k+1})-f(x_k)) + \tau h \eta_{k} (f(x_k)-f(x^\star)) \\ & \ \ \ \ \ + \frac{\eta_k^2}{2} \|\nabla f(x_k)+ \epsilon_k\|^2  - \eta_k\langle \nabla f(x_k), x_k-x^\star \rangle  - \eta_k\langle  \epsilon_k, x_k-x^\star \rangle.\\
&\le  \tau h \varphi_{k+1} (f(x_{k+1})-f(x_k)) + \frac{\eta_k^2}{2} \|\nabla f(x_k)+ \epsilon_k\|^2   - \eta_k\langle  \epsilon_k, x_k-x^\star \rangle,\\
\end{align*}
where in the second equality we added and subtracted $\tau \eta_{k} (f(x_k)-f(x^\star))$ (recall that for $k\ge0$, $h\varphi_{k+1} = \sum_{i=0}^{k}\eta_i)$ and in the second inequality the weak-quasi-convexity assumption.
Next, thanks to Lemma~\ref{lemma:discrete_time_1},
\begin{align*}
&\E[\EE(k+1)-\EE(k)]=\\  &\le  \tau h \varphi_{k+1} \E [f(x_{k+1})-f(x_k)] + \frac{\eta_k^2}{2} \E[\|\nabla f(x_k)\|^2]+ \frac{\eta_k^2}{2} \E [\|\epsilon_k\|^2 ] \nonumber\\ 
&=  \tau h \varphi_{k+1} \E [f(x_{k+1})-f(x_k)] + \frac{\eta_k^2}{2} \E [\|\nabla f(x_k)\|^2]+ \frac{\eta_k^2}{2} \E_{\mathcal{F}_{k-1}} [\|\epsilon_k\|^2 ] \nonumber\\ 
&\le \tau h \varphi_{k+1}  \E [f(x_{k+1})-f(x_k)] + \frac{\eta_k^2}{2}  \E\|\nabla f(x_k)\|^2+ \frac{\eta_k^2 d \sigma^2_*}{2 b_k} \\ 
&\le \tau h \varphi_{k+1} \left(\left(\frac{L\eta_k^2}{2}-\eta_k\right)\E\left[\|\nabla f (x_k)\|^2\right] +\frac{L d \sigma^2_* \eta_k^2 }{2 b_k}\right)+ \frac{\eta_k^2}{2} \E [\|\nabla f(x_k)\|^2]+ \frac{\eta_k^2 d \sigma^2_*}{2}\\
&\le\left( \frac{\eta_k^2}{2} + \tau \eta_0\varphi_{k+1} \left(\frac{L\eta_k^2}{2 b_k}-\eta_k\right)\right)\E\left[\|\nabla f (x_k)\|^2\right] + \frac{\eta_k^2 d \sigma^2_*(1+L\tau \varphi_{k+1})}{2 b_k}.
\end{align*}
If $h\le 2/L$, then $\frac{L\eta_k^2}{2}-\eta_k\le 0$. Moreover, under this condition, since for all $k\ge 0$ we have $\varphi_{k+1}\ge\eta_k$, it is clear that $\varphi_{k+1} \left(\frac{L\eta_k^2}{2}-\eta_k\right) \le \eta_k \left(\frac{L\eta_k^2}{2}-\eta_k\right)$. Hence

$$\E [\EE(k+1)-\EE(k)]\le \left( \frac{\eta_k^2}{2} + \tau \eta_k \left(\frac{L\eta_k^2}{2}-\eta_k\right)\right)\E\left[\|\nabla f (x_k)\|^2\right] + \frac{\eta_k^2 d \sigma^2_*(1+L\tau \varphi_{k+1})}{2 b_k}.$$

It is easy to see that $\frac{\eta_k^2}{2} + \tau \eta_k \left(\frac{L\eta_k^2}{2}-\eta_k\right)\le 0$ if and only if $h\le \frac{2\tau-1}{\tau L}$. Under this condition, since $\E\left[\|\nabla f (x_k)\|^2\right]\ge 0$,

$$\E [\EE(k+1)-\EE(k)]\le  \frac{\eta_k^2 d \sigma^2_*(1+L\tau \varphi_{k+1})}{2 b_k}.$$

 Finally, by linearity of integration,

\begin{multline*}
    \E [\EE(k+1)-\EE(0)] = \E \left[\sum_{i=0}^k \EE(i+1)-\EE(i)\right]\\ =  \sum_{i=0}^k\E[ \EE(i+1)-\EE(i)] =\frac{d \sigma^2 h^2}{2} \sum_{i=0}^{k}\frac{\psi_i^2 (1+L\tau \varphi_{i+1})}{b_i}.
\end{multline*}

The result then follows from the definition of $\EE$.
\end{proof}

The following proposition has to be compared to Thm.~\ref{prop:CONV_GF_PL} for MB-PGF.

\begin{frm-thm-app}
   Assume \textbf{(H)}, \textbf{(H$\boldsymbol{\sigma}$)}, \textbf{(H\begin{tiny}{P\L{}}\end{tiny})}. If $h\le 1/L$, then for all $k\ge 0$ we have:
\begin{multline*}
    \E\left[ (f(x_{k+1})-f(x^\star))\right]\le\\  \left(\prod_{i=0}^{k}(1-\mu \ h\psi_i)\right)(f(x_0)-f(x^\star))+ \frac{h \ d \ L \ \sigma^2_* }{2}\sum_{i=0}^k \frac{\prod_{\ell=0}^{k}(1-\mu \ h\psi_\ell)}{ \prod_{j=0}^{i}(1-\mu \ h\psi_l)} \frac{\psi_i^2}{b_i} h.
\end{multline*}
 \label{prop:CONV_GD_PL}
\end{frm-thm-app}

\begin{proof}
Starting from Lemma~\ref{lemma:discrete_time_1} we apply the P\L{} property. If $\frac{L\eta_k^2}{2}-\eta_k\le 0$, that is $\eta_k\le 2/L$ for all $k$, then
\begin{align*}
\E [f(x_{k+1})-f(x_k)]& \le \left(\frac{L\eta_k^2}{2}-\eta_k\right)\E\left[\|\nabla f (x_k)\|^2\right] +\frac{L d \sigma^2_* \eta_k^2 }{2}\\
& \le 2\mu \left(\frac{L\eta_k^2}{2}-\eta_k\right)\E[f(x_{k})-f(x^\star))] +\frac{L d \sigma^2_* \eta_k^2 }{2}
\end{align*}
Furthermore, if $\eta_k \le 1/L$ for all $k$ then $\frac{L\eta_k^2}{2} - \eta_k \le - \frac{\eta_k}{2 b_k}$:

\begin{equation}
\E [f(x_{k+1})-f(x_k)]\le -\mu \eta_k \E [f(x_k)-f(x^\star)]+\frac{L d \sigma^2_* \eta_k^2 }{2 b_k}.
\label{eq:PL_discrete_intermediate}
\end{equation}
 
Consider now the Lyapunov function inspired by the continuous time prospective (see Thm.~\ref{prop:CONV_GF_PL}):
\begin{equation*}
\EE(k) := 
\begin{cases}
\prod_{i=0}^{k-1}(1-\eta_i\mu)^{-1} (f(x_k)-f(x^\star))& k>0\\
(f(x_k)-f(x^\star)) & k=0
\end{cases}.
\end{equation*}
We have, for $k\ge 0$,
\begin{align*}
&\E [\EE(k+1)-\EE(0)]\\ &= \E \left[\sum_{i=0}^k \EE(i+1)-\EE(i)\right] \\
&=  \sum_{i=0}^k\E[ \EE(i+1)-\EE(i)] \\
&=  \sum_{i=0}^k \left( \prod_{j=0}^{i}(1-\eta_j\mu)^{-1}\right)\E\left[f(x_{i+1})-f(x^\star)- (1-\eta_{i}\mu)(f(x_i)-f(x^\star))\right] \\
&=  \sum_{i=0}^k \left( \prod_{j=0}^{i}(1-\eta_j\mu)^{-1}\right)\E\left[f(x_{i+1})-f(x_i)+\eta_{i}\mu(f(x_i)-f(x^\star))\right]. \\
\end{align*}
Using Lemma~\ref{lemma:discrete_time_1},
\begin{align*}
&\E [\EE(k+1)-\EE(0)]\\ 
&\le  \sum_{i=0}^k \left( \prod_{j=0}^{i}(1-\eta_j\mu)^{-1}\right)\left(-\mu \eta_i \E[f(x_i)-f(x^\star)]+\frac{L d \sigma_*^2 \eta_i^2 }{2 b_i}+\eta_{i}\mu\E[f(x_k)-f(x^\star)]\right) \\
&\le  \sum_{i=0}^k \left( \prod_{j=0}^{i}(1-\eta_j\mu)^{-1}\right)\frac{L d \sigma_*^2 \eta_i^2 }{2 b_i} ,\\
\end{align*}
where in the first inequality we used Eq.~\eqref{eq:PL_discrete_intermediate} .
By plugging in the definition of $\EE$, 

$$ \prod_{i=0}^{k}(1-\eta_i\mu)^{-1}  \E_{\mathcal{F}_k}\left[ (f(x_{k+1})-f(x^\star))\right]\le  f(x_0)-f(x^\star)+ \sum_{i=0}^k \left( \prod_{j=0}^{i}(1-\eta_j\mu)^{-1}\right)\frac{L d \sigma_*^2 \eta_i^2 }{2 b_i} $$

which gives the desired result.
\end{proof}
\newpage
\subsubsection{Asymptotic rates under decreasing adjustment factor}
Can be derived easily using the same arguments as in App.~\ref{sec:proofs_MB_PGF_decrease}, with the same final results.

\subsubsection{Limit sub-optimality under constant adjustment factor}
\label{sec:proofs_SGD_ball}
In this paragraph we pick $\psi_k = 1$ and $b_k = b$ for all $k$ and study the ball of convergence of SGD. The results can be found in Tb.~\ref{tb:summary_rates_constant_GD}. The only non-trivial limit is the one for P\L{} functions.

\begin{table*}
\centering
\begin{tabular}{| l | l | c |}
\hline
\textbf{Condition}& \textbf{Limit} & \textbf{Bound}
\\

\hline

\textbf{(H-)}, \textbf{(H$\boldsymbol{\sigma}$)}& $\lim_{k\to\infty}\E \left[\|\nabla f(x_{\tilde k})\|^2\right]$ &
 $\frac{L \ d \ \sigma_*^2 \ h}{b}$ \\
  & & \\

\hline
\textbf{(H-)}, \textbf{(H$\boldsymbol{\sigma}$)}, \textbf{(H\begin{tiny}{WQC}\end{tiny})}&\ $\lim_{k\to\infty}\E \left[ f(x_{\tilde k}) -f(x^\star)\right]$ & $\frac{L \ d  \ \sigma_*^2 \ h}{\tau \ b}$ \\
 & & \\
\hline
\textbf{(H-)}, \textbf{(H$\boldsymbol{\sigma}$)}, \textbf{(H\begin{tiny}{P\L{}}\end{tiny})} & $\lim_{k\to\infty}\E \left[f(x_k) -f(x^\star)\right]$  & $\frac{L \ d \ \sigma^2_* h}{2 \ \mu \ b}$ \\
 & & \\
\hline 
\end{tabular}
\caption{Ball of convergence of MB-PGF under constant $\psi_k = 1$ and $b_k = b$. For $k \ge 0$, $\varphi_{k+1}= \sum_{i=0}^k \psi_i$ and $\tilde k\in[0,k]$ is a random index picked with distribution $\psi_{j}/{\varphi_{j+1}}$ for $j \in\{0,\dots,k\}$ and $0$ otherwise.}
\label{tb:summary_rates_constant_GD}
\end{table*}

By direct calculation.

\begin{align*}
\E\left[ (f(x_{k+1})-f(x^\star))\right]&\le  (1-h\mu)^{k+1}(f(x_0)-f(x^\star))+ \frac{L d \sigma^2_* h^2}{2b}  \sum_{i=0}^k(1-h\mu)^i\\
&\le  (1-h\mu)^{k+1}(f(x_0)-f(x^\star))+ \frac{L d \sigma^2_* h^2}{2b}  \sum_{i=0}^\infty(1-h\mu)^i\\
&=  (1-h\mu)^{k+1}(f(x_0)-f(x^\star))+ \frac{L d \sigma^2_* h^2}{2hb\mu}. \\
\end{align*}
Where we used the fact that for any $\rho<1$, $\sum_{i=0}^\infty \rho_i = \frac{1}{1-\rho}$. The result then follows taking the limit.

\subsubsection{Convergence rates for VR-SGD (SVRG)}
\label{sec:rates_svrg}

\begin{frm-thm-app}
Assume \textbf{(H-)}, \textbf{(H\begin{tiny}{RSI}\end{tiny})} and choose $\xi_k = k-\sum_{j=1}^\infty\delta_{k-jm}$~(sawtooth wave), where $\delta$ is the Kronecker delta. Let $\{x_k\}_{k\ge0}$ be the solution to SGD with VR with additional jumps at times $(jm)_{j\in\N}$: we jump picking $x_{(j+1) m}$ uniformly in $\{x_k\}_{jm\le k< (j+1) m}$. Then,
\vspace{-2mm}
$$\E[\|x_{jm}-x^\star\|^2]=\left(\frac{1+2L^2h^2 m}{h m (\mu-3L^2 h)}\right)^j \|x_0-x^*\|^2.$$
\label{prop:CONV_VR_RSI_discrete}
\vspace{-3mm}
\end{frm-thm-app} 
\begin{proof} Start by computing
\begin{align*}
    &\frac{1}{2}\E\left[\|x_{k+1}-x^\star\|^2\right]\\
    &= \frac{1}{2}\E\left[\|x_k -x^\star - h \mathcal{G}_{\text{VR}}(k)\|^2\right]\\
    &= -h\E\left[\langle\nabla f(x_k), x_k-x^\star\rangle\right] +L^2 h^2 \E[\|\mathcal{G}_{\text{VR}}(k)\|^2]
    \end{align*}
    where we used the fact that $\mathcal{G}_{\text{VR}}$ is unbiased. Consider iterations $jm\le k\le j(m+1)$. Our choice of $\xi$ fixes the pivot to $x_{jm}$. Using smoothness, Cor.~\ref{cor:STOC_SVRG_var_QG} and the restricted-secant-inequality, we get,
    \begin{align*}
    &\frac{1}{2}\E\left[\|x_{k+1}-x^\star\|^2\right] -  \frac{1}{2}\E\left[\|x_k -x^\star\|^2\right]\\
    &\le -h\E\left[\langle\nabla f(x_k), x_k-x^\star\rangle\right] + L^2 h^2 \E\left[\|x_k-x^\star\|^2\right] + L^2 h^2 \E\left[\|x_{jm}-x^\star\|^2\right]\\    
&\le -\frac{h\mu}{2}\E\left[\|x_{k}-x^\star\|^2\right] + 2L^2 h^2 \E\left[\|x_k-x^\star\|^2\right] + L^2 h^2 \E\left[\|x_{jm}-x^\star\|^2\right]\\    
&= - \frac{h}{2}\left(\mu-2L^2 h\right)\E\left[\|x_{k}-x^\star\|^2\right]+ L^2 h^2 \E\left[\|x_{jm}-x^\star\|^2\right].\\ 
\end{align*}

Finally, summing from $jm$ to $j(m+1)$, we have
\begin{multline*}
    \frac{1}{2}\E\left[\|x_{j(m+1)}-x^\star\|^2\right] -  \frac{1}{2}\E\left[\|x_{jm} -x^\star\|^2\right] \\ \le- \frac{hm}{2}\left(\mu-2L^2 h\right)\frac{1}{m} \sum_{k=jm}^{jm+m-1}\E\left[\|x_{k}-x^\star\|^2\right]+ L^2 h^2 m \E\left[\|x_{jm}-x^\star\|^2\right].
\end{multline*}

Therefore, dropping the first term, 
$$\E\left[\|x_{(j+1)m}-x^\star\|^2\right] = \frac{1}{m}\sum_{k=jm}^{jm+m-1}\E\left[\|x_{k}-x^\star\|^2\right].$$
Redefining (\textit{jumping} to) $x_{j(m+1)} \sim \mathcal{U}(\{x_k\}_{jm\le k\le (j+1)m})$, we get $$\E[\|x_{j(m+1)}-x^\star\|^2]\le\frac{1+2L^2h^2 m}{h m (\mu-2L^2 h)}\E\left[\|x_{jm} -x^\star\|^2\right].$$
\end{proof}

\section{Time stretching}
\label{sec:time-change-proof}
\begin{frm-thm-app}[Restated Thm.~\ref{thm:time-change}]
Let $\{X(t)\}_{t\ge0}$ satisfy PGF and define $\tau(\cdot) = \varphi^{-1}(\cdot)$, where $\varphi(t) = \int_0^t\psi(s)ds$. For all $t\ge 0$, $X\left(\tau(t)\right) = Y(t)$ in distribution, where $\{Y(t)\}_{t\ge0}$ satisfies
$$dY(t)=-\nabla f(Y(t)) dt + \sqrt{\frac{h \  \psi(\tau(t))}{m(\tau(t))}} \sigma(\tau(t)) \ d\tilde B(t),$$
where $\{\tilde B(t)\}_{t\ge0}$ is a Brownian Motion.
\label{thm:time-change-app}
\end{frm-thm-app}

\begin{proof} By definition, $X(t)$ is such that
$$X(t) = -\int_0^t \psi(r)\nabla f(X(r)) \ dr + \int_0^t \psi(r)\sqrt{\frac{h}{m(r)}} \sigma(r) \ dB(r).$$
Therefore
$$X(\tau(t)) = -\underbrace{\int_0^{\tau(t)} \psi(r)\nabla f(X(r)) \ dr}_{:=A} + \underbrace{\int_0^{\tau(t)} \psi(r)\sqrt{\frac{h}{m(r)}} \sigma(r) \ dB(r)}_{:=B}.$$
Using the change of variable formula for Riemann integrals, we get
\begin{multline*}
    A =\int_0^{\tau(t)} \psi(r)\nabla f(X(r))dr = -\int_0^{t} \tau'(r) \cdot \psi(\tau(r)) \cdot  \nabla f(X(\tau(r))) dr=\\
    =\int_0^{t} \frac{\cancel{\psi(\tau(r))}}{\cancel{\psi(\tau(r))}}\nabla f(X(\tau(r))) dr.
\end{multline*}
Using the time change formula (Thm.~\ref{thm:oxsendall}) for stochastic integrals, with $v(r):= \psi(r)\sqrt{\frac{h}{m(r)}} \sigma(r)$,
\begin{multline*}
    B=\int_0^{\tau(t)} \psi(r)\sqrt{\frac{h}{m(r)}} \sigma(r) \ dB(r) = \int_0^{t} \frac{\psi(\tau(r))}{\sqrt{\tau'(r)}}\sqrt{\frac{h}{m(\tau(r))}} \sigma(\tau(r)) \ d\tilde B(r)=\\
    =\int_0^{t} \sqrt{\frac{h \  \psi(\tau(r))}{m(\tau(r))}} \sigma(\tau(r)) \ d\tilde B(r).
\end{multline*}

All in all, we have found that

$$X(\tau(t))=-\int_0^{t} \nabla f(X(\tau(r))) dr + \int_0^{t} \sqrt{\frac{h \  \psi(\tau(r))}{m(\tau(r))}} \sigma(\tau(r)) \ d\tilde B(r).$$

By Def.~\ref{def:SDE_ito_process}, this is equivalent to saying that $Y := X \circ \tau$ satisfies the differential in the theorem statement.
\end{proof}